\documentclass[11pt, a4paper]{amsart}
\usepackage{a4wide}
\usepackage{enumitem}
\usepackage{ulem}
\usepackage{amssymb}
\DeclareSymbolFontAlphabet{\amsmathbb}{AMSb}%
\usepackage{mathbbol}
\usepackage{latexsym}

\usepackage[pdftex,
pdftitle={Approximation of SPDE bridges},
bookmarksopen,
colorlinks,
linkcolor=black,
urlcolor=black,
citecolor=black
]{hyperref}
\hypersetup{pdfauthor={G. di Nunno, S. Ortiz--Latorre, and A. Petersson}}

\usepackage{dsfont}

\allowdisplaybreaks
\usepackage{mathtools}

\newcommand{\IP}{\amsmathbb{P}}
\newcommand{\R}{\amsmathbb{R}}
\newcommand{\N}{\amsmathbb{N}}
\newcommand{\cA}{\mathcal{A}}
\newcommand{\cC}{\mathcal{C}}
\newcommand{\cD}{\mathcal{D}} 
\newcommand{\cE}{\mathcal{E}}
\newcommand{\cF}{\mathcal{F}}
\newcommand{\cI}{\mathcal{I}}
\newcommand{\cJ}{\mathcal{J}}
\newcommand{\cK}{\mathcal{K}}
\newcommand{\cL}{\mathcal{L}}
\newcommand{\cM}{\mathcal{M}}
\newcommand{\Op}{\operatorname{O}}
\DeclareMathOperator{\E}{\amsmathbb{E}}
\DeclareMathOperator{\Cov}{\mathsf{Cov}}
\DeclareMathOperator{\trace}{Tr}
\newcommand{\dd}{\,\mathrm{d}}
\newcommand{\dom}{\mathrm{dom}} 
\newcommand{\lrinpro}[3][{}]{ \left\langle #2 , #3 \right\rangle_{#1} }
\newcommand{\inpro}[3][{}]{ \langle #2 , #3 \rangle_{#1} }
\newcommand{\dualp}[3][{}]{ {}_{#1}\langle #2 , #3 \rangle_{#1^*} }
\newcommand{\norm}[2][{}]{\| #2 \|_{#1}}
\newcommand{\lrnorm}[2][{}]{\left\| #2 \right\|_{#1}}
\newcommand{\Bignorm}[2][{}]{\Big\| #2 \Big\|_{#1}}

\newtheorem{lemma}{Lemma}[section]
\newtheorem{proposition}[lemma]{Proposition}

\newtheorem{theorem}[lemma]{Theorem}

\theoremstyle{remark}
\newtheorem{remark}[lemma]{Remark}
\newtheorem{proofpart}{Part}

\makeatletter
\@addtoreset{proofpart}{lemma}
\makeatother

\theoremstyle{definition}

\newtheorem{assumption}[lemma]{Assumption}
\newtheorem{example}[lemma]{Example}

\makeatletter
\@namedef{subjclassname@1991}{{\normalfont 2020} Mathematics Subject Classification}
\makeatother

\title[Approximation of SPDE bridges]{SPDE bridges with observation noise and their spatial approximation}
\date{\today}

\author[G.~di~Nunno]{Giulia di Nunno} \address[Giulia di Nunno]{\newline The Faculty of Mathematics and Natural Sciences
	\newline Department of Mathematics
	\newline Postboks 1053
	\newline Blindern
	\newline 0316 Oslo, Norway.} \email[]{giulian@math.uio.no}

\author[S.~Ortiz-Latorre]{Salvador Ortiz--Latorre} \address[Salvador Ortiz-Latorre]{\newline The Faculty of Mathematics and Natural Sciences
	\newline Department of Mathematics
	\newline Postboks 1053
	\newline Blindern
	\newline 0316 Oslo, Norway.} \email[]{salvadoo@math.uio.no}

\author[A.~Petersson]{Andreas Petersson} \address[Andreas Petersson]{\newline The Faculty of Mathematics and Natural Sciences
	\newline Department of Mathematics
	\newline Postboks 1053
	\newline Blindern
	\newline 0316 Oslo, Norway.} \email[]{andreep@math.uio.no}

\thanks{\textit{Acknowledgement.} The research leading to these results was funded within the project STORM: Stochastics for Time-Space Risk Models, from the Research Council of Norway (RCN), project number: 274410. The authors wish to express many thanks to an anonymous referee who helped to improve the results and the presentation.}


\begin{document}

\begin{abstract}
This paper introduces SPDE bridges with observation noise and contains an analysis of their spatially semidiscrete approximations. The SPDEs are considered in the form of mild solutions in an abstract Hilbert space framework suitable for parabolic equations. They are assumed to be linear with additive noise in the form of a cylindrical Wiener process. The observational noise is also cylindrical and SPDE bridges are formulated via conditional distributions of Gaussian random variables in Hilbert spaces. A general framework for the spatial discretization of these bridge processes is introduced. Explicit convergence rates are derived for a spectral and a finite element based method. It is shown that for sufficiently rough observation noise, the rates are essentially the same as those of the corresponding discretization of the original SPDE.
\end{abstract}

	\maketitle

\section{Introduction}
%\label{sec:introduction}

In the geophysical sciences, phenomena that involve time and space are ubiquitous and models describing their behavior grow, out of necessity, ever more complex. A common way of capturing this complexity is to include randomness in the models. This could, for instance, model unknown past sources, measurement errors or effects on different physical scales. A popular class of such models consist of stochastic partial differential equations (SPDEs). Apart from the geophysical sciences, such models occur in finance \cite{F01}, cell biology \cite{PFABS21} and many other areas \cite{DPZ14,LR17}.  

In this paper we consider, in an abstract setting, an example of an SPDE that captures many important characteristics of more complex models: a linear stochastic reaction-diffusion equation on a bounded domain $\cD \subset \R^d$, $d=1, 2, 3$. This is given by
\begin{equation}
	\label{eq:react-diff}
	\begin{alignedat}{2}		
		\frac{\partial X}{\partial t}(t,\xi) + \cA X(t,\xi) &= \frac{\partial W}{\partial t}(t,\xi) &&\text{ for } t \in (0,T], \xi \in \cD, T < \infty, \\
		X(0,\xi) &= x(\xi),   &&\text{ for } \xi \in \cD.
	\end{alignedat}
\end{equation}
Here $\cA = -\sum^{d}_{i,j=1} \frac{\partial}{\partial \xi_i} a_{i,j}  \frac{\partial}{\partial \xi_j} + a_0$ is an elliptic operator with sufficiently smooth coefficients $(a_{i,j})_ {i, j=1^d}$, $a_0$ along with suitable boundary conditions. The noise term denoted by $\partial W / \partial t$ is Gaussian, white in time and (possibly) correlated in space, and $x$ is the initial state of the SPDE. This can be seen as a simplification of SPDEs used to model, e.g., the dynamics of sea surface temperature anomalies \cite{HH87, P89, PR96}. We understand this equation as an infinite-dimensional stochastic differential equation (SDE) of It\^o type in the framework of~\cite{DPZ14}. We consider its mild solution 
\begin{equation}
	\label{eq:mild-sol-intro}
	X^x(t) := S(t) x + \int^t_0 S(t-s) \dd W(s), t \in [0,T],
\end{equation}
as a stochastic process $X^x$ taking values in the Hilbert space $H = \cL^2(\cD)$. Here, $W$ is a Wiener process $W$ in $H$ with spatial covariance (operator) $Q$. The operator valued function $S$ is an analytic $C_0$-semigroup generated by $-A$. In the setting of~\eqref{eq:react-diff}, $A$ is the operator $\cA$ regarded as an unbounded linear operator on $H$. 

We are concerned with the SPDE bridge $X^{x,y}$ with observation noise associated to~\eqref{eq:react-diff}. Formally, this is the process $X^x$ conditioned on having observed $X^x(T) + Z = y$, where $Z$ is a Gaussian noise (the observation noise) in $H$, independent of $W$. We rigorously define such bridge processes in a framework inspired by~\cite{S91,GM06,GM08}. Therein, the authors consider SPDE bridges without observation noise, also called pinned Ornstein--Uhlenbeck processes in infinite dimensions. They use the fact that $(X^x,X^x(T))$ is a pair of jointly Gaussian random variables in $L^2([0,T],H) \oplus H$ to derive an expression for $X^x$ conditioned on $X^x(T) = y$, see Section~\ref{sec:spde-bridges} below for details. The resulting process is used to investigate the Markov transition semigroup for a semilinear SPDE. Our first goal with this paper is to provide a rigorous theoretical framework for SPDE bridges that includes observation noise.

For finite-dimensional SDEs, bridge processes are well understood. They are, for example, used in solving inverse problems related to SDEs, such as parameter estimation in a Bayesian setting. This consists of sampling SDE bridges, either with or without observation noise, and employing these in a Markov Chain Monte Carlo algorithm \cite{SMZ17}. Our work can be seen as a stepping stone towards such parameter estimation in an SPDE setting. In order to sample SDE bridges, they must be approximated numerically. More precisely, they are discretized in time, resulting in an error. This can be quantified (see, e.g., \cite{PRS13}), making sure that it does not affect the overall convergence of an algorithm such as the ones employed in \cite{SMZ17}. If we want to consider an analogous approach to parameter estimation in an SPDE setting, approximation errors have to be quantified for both a temporal and a spatial discretization. Our second, and major, goal with this paper is to analyze spatially semidiscrete approximations of SPDE bridges with observation noise. This is the discretization direction unique to SPDEs. 

Our extension of the framework of~\cite{GM06,GM08} to include observation noise provides three advantages. First, it is a more natural setting for the inclusion of measurement errors in the inverse problem, if we think of the observation of $X^x(T)$ as a measurement. Moreover, it allows for more general covariance operators $Q$, since these do not have to be assumed to be injective, as in~\cite{S91,GM06,GM08}. Finally, it allows us to derive convergence rates for the spatial approximations of $X^{x,y}$ that we consider. We consider both spectral and finite element type methods, cf.\ \cite{JK09, K14,LPS14}. Applied to forward problems (i.e., approximating $X^x$), the theory of such spatial approximations is by now relatively mature. In contrast, it is still relatively rare to find numerical analyses of methods for solving inverse problems (such as parameter estimation) related to~\eqref{eq:mild-sol-intro} \cite{C18}. In particular, to the best of our knowledge, the derivation of convergence rates for spatial approximations of SPDE bridges (with our without observation noise) has not been considered before in the literature.

The key assumption we have to make in order to obtain our result is that the covariance of $Z$, denoted $\tilde Q$, commutes with the orthogonal projection associated to the given discretization method we consider (equivalently, $\tilde Q$ is invariant on the corresponding discretization space, see Remark~\ref{rem:commutativity}). As a special case, this includes the setting that $\tilde Q = \epsilon I$ for some $\epsilon > 0$, with $I$ denoting the identity operator. One of the main results of our paper is that in this case, the convergence rate, measured uniformly in time, for an SPDE bridge with spatially white observation noise is, asymptotically, no worse than the convergence rate for the corresponding discretization of the mild solution~\eqref{eq:mild-sol-intro}. See Remark~\ref{rem:spectral-unconditioned-rate} for the spectral method and Remark~\ref{rem:fem-unconditioned-rate} for the finite element method.

Next, we describe the structure of the paper. In Section~\ref{sec:preliminaries} we give the necessary mathematical background to our problem. We introduce Gaussian random variables in Banach spaces and cylindrical Gaussian random variables in Hilbert spaces. In particular, we focus on the notions of $\gamma$-radonifying measures and the reproducing kernel Hilbert spaces associated to Gaussian measures. We also reiterate some well-known results regarding the conditional distribution of Hilbert space valued Gaussian random variables with respect to another such variable. Thereafter, we describe the functional analytic framework we employ. In Section~\ref{sec:spde-bridges} the SPDE we consider is introduced, along with its mild solution. Thereafter, we introduce the bridge process with observation noise $Z$ that is our main object of study. The noise $Z$ is considered in an abstract Wiener space framework by treating it as a random variable in a possibly negative order Hilbert space. Along with the Gaussian law of the data involved, this leads to an explicit expression of the SPDE bridge with observation noise:
\begin{equation*}
	X^{x,y}(t) := X^x(t) - K(t) ((Q(T) + \tilde Q) A^{\eta})^{-\frac{1}{2}} (X^x(T) + Z - y).
\end{equation*}
Here $Q(t)$ is the covariance of $X(t)$, $K(t)$ is a mapping related to a formula for the conditional expectation of Gaussian random variables in Hilbert spaces and $A^{\eta}$ is a fractional power of $A$. Two facts make the analysis particularly challenging. First: $((Q(T) + \tilde Q) A^{\eta})^{-\frac{1}{2}}$ is necessarily an unbounded operator on the space on which $X^x(T) + Z$ is considered as a random variable. Second: $X^{x,y}(t)$ is only well-defined for a.e.\ $y$ on a typically negative order space. In Section~\ref{sec:approximation} we introduce an abstract spatially semidiscrete approximation $X_V^{x,y}$ of $X^{x,y}$, which takes values in a finite-dimensional subspace $V \subset H$, and we derive two key lemmas related to the approximation of conditional expectations of $X^0$. These are applied to concrete spatial approximations in the final two sections of the paper. In Section~\ref{sec:spectral-approximation} we set $V = V_N := \mathrm{span}\{e_1, \ldots, e_N\}$, where $(e_j)_{j=1}^\infty$ is the eigenbasis associated to $A$, to obtain a spectral approximation of $X^{x,y}$. In Section~\ref{sec:approximation-fem} we set $V = V_h$, a space of piecewise linear polynomials, and obtain a finite element approximation of $X^{x,y}$. We derive convergence rates for both types of approximations in the space $L^p(\Omega,\cC([0,T],H))$ for $p \ge 1$, as either $N \to  \infty$ or $h \to 0$. 

We adopt the notion of generic constants, which may vary from occurrence to occurrence and are independent of any parameter of interest, such as spatial step sizes. By the notation $a \lesssim b$ we indicate the existence of a generic constant such that $a \le C b$.
	
\section{Preliminaries}
\label{sec:preliminaries}

\subsection{Gaussian random variables in infinite dimensions}
%\label{subsec:probability-in-inf-dim}

In this section we review basic results for covariance operators, conditional distributions and Gaussian random variables in  possibly infinite-dimensional spaces, starting with some operator theory. All Banach spaces appearing in this paper are real and separable, so we assume this for the Banach spaces in this section. Let $(U_1,U_2)$ be a pair of such spaces with norm $\norm[U_j]{\cdot}$ and duality pairing $\dualp[U_j]{\cdot}{\cdot}$, $j = 1, 2$. We denote by $(\cL(U_1,U_2), \norm[\cL(U_1,U_2)]{\cdot})$ the space of bounded linear operators from $U_1$ to $U_2$ equipped with the operator norm and write $\cL(U)$ for this space when $U_1 = U_2 = U$. When $(U_1,U_2) = (H_1,H_2)$ are Hilbert spaces with inner products $\inpro[H_j]{\cdot}{\cdot}$, $j = 1, 2$, we write $\cL_1(H_1,H_2)$ and $\cL_2(H_1,H_2)$ for the subspaces of trace class (or nuclear) and Hilbert--Schmidt operators, respectively. For $H_1 = H_2 = H$ we use the shorthand notations $\cL_1(H)$ and $\cL_2(H)$.
An operator $\Gamma \in \cL_1(H_1,H_2)$ if there are two sequences $(a_j)_{j=1}^\infty \subset H_1$, $(b_j)_{j=1}^\infty \subset H_2$ such that 
\begin{equation*}
	%\label{eq:trace_characterization}
	\Gamma x = \sum_{j = 1}^{\infty} \inpro[H_1]{x}{a_j} b_j  
\end{equation*} 
for all $x \in H_1$ and 
\begin{equation*} 
	\sum_{j = 1}^{\infty} \norm[H_1]{a_j} \norm[H_2]{b_j} < \infty. 
\end{equation*}
The space $\cL_1(H_1,H_2)$ is a separable Banach space with norm
\begin{equation*}
	%\label{eq:trace_norm}
	\norm[\cL_1(H_1,H_2)]{\Gamma} := \inf_{\substack{(a_j) \subset H_1 \\ (b_j) \subset H_2}} \left\{ \sum_{j = 1}^{\infty} \norm[H_1]{a_j} \norm[H_2]{b_j} : \Gamma = \sum_{j = 1}^{\infty} \inpro[H_1]{\cdot}{a_j} b_j \right\},
\end{equation*}
see~\cite[Appendix~B]{PR07}, while $\cL_2(H_1,H_2)$ is a separable Hilbert space with inner product
\begin{equation*}
	\inpro[\cL_2(H_1,H_2)]{\Gamma_1}{\Gamma_2} := \sum_{j = 1}^{\infty} \inpro[H_2]{\Gamma_1 e_j}{\Gamma_2 e_j}
\end{equation*}
for $\Gamma_1, \Gamma_2 \in \cL_2(H_1,H_2)$ and an arbitrary orthonormal basis $(e_j)_{j = 1}^\infty$ of $H_1$. 
We have $\Gamma \in \cL_i(H_1,H_2)$ if and only if $\Gamma^* \in \cL_i(H_2,H_1)$ with 
\begin{equation*}
	%\label{eq:schattenadjoint}
	\norm[\cL_i(H_1,H_2)]{\Gamma} = \norm[\cL_i(H_2,H_1)]{\Gamma^*}
\end{equation*}
for $i \in \{1,2\}$. Here $\Gamma^*$ denotes the adjoint of $\Gamma$. Let $(H_3,H_4)$ be another pair of Hilbert spaces. If $\Gamma_1 \in \cL(H_1,H_2)$, $\Gamma_3 \in \cL(H_3,H_4)$ and $\Gamma_2 \in \cL_i(H_2,H_3)$ for some $i \in \{1,2\}$, then $\Gamma_3 \Gamma_2 \Gamma_1 \in \cL_i(H_1,H_4)$ and 
\begin{equation}
	\label{eq:schatten-bound-1}
	\norm[\cL_i(H_1,H_4)]{\Gamma_3 \Gamma_2 \Gamma_1} \le  \norm[\cL(H_3,H_4)]{\Gamma_3} \norm[\cL_i(H_2,H_3)]{\Gamma_2} \norm[\cL(H_1,H_2)]{\Gamma_1}.
\end{equation}
Moreover, if $\Gamma_1 \in \cL_2(H_1,H_2)$ and $\Gamma_2 \in \cL_2(H_2,H_3)$, then $\Gamma_2 \Gamma_1 \in \cL_1(H_1,H_3)$ and
\begin{equation}
	\label{eq:schatten-bound-2}
	\norm[\cL_1(H_1,H_3)]{\Gamma_2 \Gamma_1} \le \norm[\cL_2(H_2,H_3)]{\Gamma_2} \norm[\cL_2(H_1,H_2)]{\Gamma_1}. 
\end{equation}

For $\Gamma \in \cL(H_1)$, the operator range $\Gamma(H_1)$ is a Hilbert space when equipped with the inner product $\inpro[\Gamma(H_1)]{\cdot}{\cdot} := \inpro[H_1]{\Gamma^{-1}\cdot}{\Gamma^{-1}\cdot}$. Here $\Gamma^{-1}$ is the pseudo-inverse of $\Gamma$, i.e., $\Gamma^{-1} v$ is for $v \in \Gamma(H_1)$ the unique element $w \in \ker(\Gamma)^\perp$, the orthogonal complement of $\ker(\Gamma)$, fulfilling $\Gamma w = v$, see~\cite[Appendix~C]{PR07}. For $\Gamma \in \cL_1(H_1)$, the trace of $\Gamma$ is defined by $\trace(\Gamma) := \sum^\infty_{j=1} \inpro[H_1]{\Gamma e_j}{e_j},$ for $\Gamma \in \cL_1(H_1)$ and an arbitrary orthonormal basis $(e_j)_{j=1}^\infty$. It holds that $|\trace(\Gamma)| \le \norm[\cL_1(H_1)]{\Gamma}$,
with equality if and only if $\Gamma \in \Sigma^+(H_1)$, the space of positive semidefinite symmetric operators on $H_1$.

Let now $(\Omega, \cA, (\cF_t)_{t \in [0,T]}, \IP)$ be a complete filtered probability space satisfying the usual conditions, which is to say that $\cF_0$ contains all $\IP$-null sets and $\cF_t = \cap_{s > t} \cF_s$ for all $t \in [0,T]$. A measurable mapping $X \colon \Omega \to U$ is called a $U$-valued random variable. Since $U$ is separable, this is equivalent to requiring that $X$ is strongly measurable \cite[Corollary~1.1.10]{HvNVW17a}. By $L^p(\Omega,U)$, $p \in [1, \infty)$, we denote the Banach space of all $U$-valued random variables~$X$ with finite norm $\norm[L^p(\Omega,U)]{X} := (\E[\norm[U]{X}^p])^{1/p}$. 
In the case of Hilbert space-valued random variables $X \in L^2(\Omega,H_1)$, $Y \in L^2(\Omega,H_2)$, we define the cross-covariance (operator) $\Cov(X,Y) \colon H_2 \to H_1 $ of $X$ and $Y$ by 
\begin{equation*}
	\Cov(X,Y) u := \E[\inpro[H_2]{(Y - \E[Y])}{u} (X - \E[X])]
\end{equation*}
and the covariance (operator) of $X$ by $\Cov(X) := \Cov(X,X)$.
Both operators are of trace class \cite{B00}. In particular, $\norm[\cL_1(H_1)]{\Cov(X)} = \trace(\Cov(X)) = \E[\norm[H_1]{X-\E[X]}^2]$ whenever either of these quantities are finite. Moreover, if $(e_j)_{j=1}^\infty$ is an arbitrary orthonormal basis of $H_1$ we have by~\cite[Proposition~I.1.10]{VTC87} that the $\sigma$-algebra $\sigma(X)$ generated by $X$ fulfills
\begin{equation}
	\label{eq:sigma-algebra-X}
	\sigma(X) = \sigma((\inpro[H_1]{X}{e_j})_{j=1}^\infty).
\end{equation} 

A Banach-space valued random variable $X \colon \Omega \to U$ is said to be Gaussian if $\dualp[U]{X}{u}$ is a Gaussian real-valued random variable for all $u$ in the continuous dual space $U^*$. Then $X \in L^p(\Omega,U)$ for all $p \ge 1$. Equivalently, the image measure $\IP \circ X^{-1}$ is Gaussian on $U$.
For such random variables, we define a covariance operator $\Cov(X) \colon U^* \to U$ by $\Cov(X)u := \E[\dualp[U]{X-\E[X]}{u} (X-\E[X])]$ for $u \in U^*$. Let us denote by $H_X \subset U$ the reproducing kernel Hilbert space of the Banach-space valued random variable $X$. This space is given by the completion of $\Cov(X)(U^*)$ with respect to the norm induced by the inner product 
$
\inpro[H_X]{\Cov(X) u}{\Cov(X) v} := \dualp[U]{\Cov(X)u}{v}.
$
When $U=H_1$ is a Hilbert space, the definition of $\Cov(X)$ coincides with the previously given definition. Since the operator $\Cov(X) \in \Sigma^+(H)$, it has a unique square root $\Cov(X)^{1/2} \in \Sigma^+(H)$ and $H_X = \Cov(X)^{1/2}(H_1)$. If $H_2 \subset H_1$ is a subspace, continuously embedded into $H_1$, and if $\norm[L^2(\Omega,H_2)]{X} < \infty$, then the covariance $\Cov_{H_2}(X)$ of $X$ as an $H_2$-valued random variable is related to $\Cov(X) = \Cov_{H_1}(X)$ through the identity
\begin{equation}
	\label{eq:covariance-in-subspace}
	\Cov(X) = I_{H_2 \hookrightarrow H_1} \Cov_{H_2}(X) I^*_{H_2 \hookrightarrow H_1}.
\end{equation}
This follows from the fact that $\Cov_{H_1}(X)^{1/2}(H_1) = \Cov_{H_2}(X)^{1/2}(H_2)$, along with the factorization $\Cov(X) = I_{\Cov(X)^{1/2}(H_1)\hookrightarrow H_1} I^*_{\Cov(X)^{1/2}(H_1)\hookrightarrow H_1}$ \cite{R11}. For more details on Banach-space valued Gaussian random variables, see, e.g.,~\cite{VTC87, HvNVW17b}.

We say that $X$ is a cylindrical random variable in a Hilbert space $H$, if it is a linear map $X$ from $H$ into the space of real-valued random variables. This is a generalization of the notion of an $H$-valued random variable $Y$, since we may interpret $Y$ as a cylindrical random variable $X$ by identifying $X$ with the random linear functional $\inpro[H]{Y}{\cdot}$. If there, for a given cylindrical random variable $X$, exists such a $Y$, $X$ is said to be induced by $Y$. We say that $X$ is a strongly Gaussian cylindrical random variable with covariance $\Cov(X)$ and zero mean if $X(u)$ is a real-valued Gaussian random variable with zero mean for all $u \in H$ and there is an operator $\Cov(X) \in \Sigma^+(H)$ such that $\E[X(u)X(v)] = \inpro[H]{\Cov(X) u}{v}$ for all $u,v \in H$. For all operators in $\Sigma^+(H)$, a linear map with the required properties can be constructed. Since we will not deal with so called weakly Gaussian cylindrical random variables (see~\cite{R11} for this, as well as a general introduction to cylindrical measures), from here on we omit the word strongly. Moreover, we introduce the formal notation $\inpro[H]{X}{u} : = X(u)$ for all $u \in H$. 
\begin{remark}
	\label{rem:cyl-notation}
	The notation $\inpro[H]{X}{u} : = X(u)$ for a cylindrical random variable is only formal. We emphasize that $X \notin H$ in a mean-square sense.
\end{remark}
We note that 
\begin{equation*}
	\E[\inpro[H]{X}{u-v}^2] = \norm[H]{\Cov(X)^{\frac{1}{2}} (u-v)}^2 \le \norm[\cL(H)]{\Cov(X)^{\frac{1}{2}}}^2 \norm[H]{u-v}^2
\end{equation*}
for $u,v \in H$. This is to say that $\inpro[H]{X}{\cdot}$ is continuous on $H$ with respect to $L^2(\Omega,\R)$.
When $\Cov(X) = I$, $X$ is said to be a standard cylindrical random variable in $H$.
An operator $\Gamma \colon H \to U$, where $U$ is a Banach space, is said to be $\gamma$-radonifying if, for a standard cylindrical random variable $X$ in $H$, there is a $U$-valued Gaussian random variable $Y$ such that $\dualp[U]{Y}{u} = \inpro[H]{X}{\Gamma^* u}$ for all $u \in U^*$. The space $\gamma(H,U)$ of $\gamma$-radonifying operators is a Banach space with norm
\begin{equation}
	\label{eq:gamma-rad-norm}
	\norm[\gamma(H,U)]{\Gamma}^2 := \E\Big[\Bignorm[U]{\sum_{j = 1}^\infty z_j  \Gamma e_j}^2 \Big],
\end{equation}
where $(z_j)_{j=1}^\infty$ is a sequence of iid Gaussian random variables and $(e_j)_{j=1}^\infty$ is an arbitrary orthonormal basis of $H$ \cite[Corollary~3.21]{VN10}. The sum inside this expression converges in $L^2(\Omega,U)$ if and only if $\Gamma \in \gamma(H,U)$.
When $H=H_1$ and $U=H_2$ is a Hilbert space, $\gamma(H_1,H_2) = \cL_2(H_1,H_2)$ with equal norms \cite[Proposition~13.5]{VN10}. Like the space of Hilbert--Schmidt operators, the space $\gamma(H,U)$ is an operator ideal \cite[Theorem~6.2]{VN10}. If $X$ is a $U$-valued Gaussian random variable, then, by \cite[Proposition~8.6]{VN10},  
\begin{equation}
	\label{eq:gamma-rad-expectation-property}
	\E\Big[ \norm[U]{X}^2 \Big] = \norm[\gamma(H_X,U)]{I_{H_X \hookrightarrow U}}^2 < \infty.
\end{equation}
In fact, all moments of $X$ are finite and can be bounded by a  constant times the $L^2(\Omega,U)$-norm \cite[Proposition~3.14]{H09}.

Gaussian cylindrical random variables $X$ on $H_1$ may be regarded as Gaussian $H_2$-valued random variables, for some larger Hilbert space $H_2 \supset H_1$. We can always construct a real separable Hilbert space $H_2 \hookleftarrow H_1$, where the embedding is dense and continuous, such that $I_{Q^{1/2}(H_1) \hookrightarrow H_2} \in \cL_2(Q^{1/2}(H_1),H_2)$, c.f.\ 
\cite[Remark~2.5.1]{PR07}. Let $H_2$ be a Hilbert space with this property. By~\cite[Lemma~VI.1.8]{DS58}, the closure of $\Cov(X)(H_1)$ in $\Cov(X)^{1/2}(H_1)$ is the set of all vectors that are perpendicular to $\{v \in \Cov(X)^{1/2}(H_1) : \Cov(X)^{1/2} v = 0 \}$ in $\Cov(X)^{1/2}(H_1)$. But since $\inpro[H_1]{\Cov(X)v}{v} = \norm[H_1]{\Cov(X)^{1/2}v}^2$ for $v \in H_1$, $\ker(\Cov(X))=\ker(\Cov(X)^{1/2})$. Therefore $\{v \in \Cov(X)^{1/2}(H_1) : \Cov(X)^{1/2} v = 0 \} = \{0\}$ so that $\Cov(X)(H_1)$ is dense in $\Cov(X)^{1/2}(H_1)$. Since $H_1$ is separable, so is $\Cov(X)^{1/2}(H_1)$. We may therefore pick an orthonormal basis $(e_j)^\infty_{j=1}$ of $\Cov(X)^{1/2}(H_1)$ such that $e_j \in \Cov(X)(H_1)$ for all $j \in \N$. Hence, we may define
\begin{equation}
	\label{eq:trace-version-of-cylindrical-rv}
	\tilde X := \sum_{j = 1}^{\infty} \inpro[H_1]{X}{\Cov(X)^{-1} e_j} e_j.
\end{equation}
Note that, for any $N \in \N$,
\begin{align*}
	&\E\left[\Bignorm[H_2]{\sum_{j = 1}^{N} \inpro[H_1]{X}{\Cov(X)^{-1} e_j} e_j}^2\right] \\ &\quad= \sum_{i,j = 1}^{N} \E\left[\inpro[H_1]{X}{\Cov(X)^{-1} e_i} \inpro[H_1]{X}{\Cov(X)^{-1} e_j}\right] \inpro[H_2]{e_i}{e_j} \\
	&\quad= \sum_{i,j = 1}^{N} \inpro[H_1]{\Cov(X)^{-\frac{1}{2}}e_i}{\Cov(X)^{-\frac{1}{2}}e_j} 
	\inpro[H_2]{e_i}{e_j} \\
	&\quad= \sum_{i,j = 1}^{N} \inpro[\Cov(X)^{\frac{1}{2}}(H_1)]{e_i}{e_j} 
	\inpro[H_2]{e_i}{e_j} = \sum_{i = 1}^{N} \norm[H_2]{e_i}^2 \le \norm[\cL_2(\Cov(X)^{\frac{1}{2}}(H_1),H_2)]{I_{\Cov(X)^{\frac{1}{2}}(H_1) \hookrightarrow H_2}}^2,
\end{align*}
so that $\tilde X$ is indeed well-defined in $L^2(\Omega,H_2)$. Since $(\inpro[H_1]{X}{\Cov(X)^{-1} e_j})^\infty_{j=1}$ is a sequence of independent and Gaussian random variables, $\tilde X$ is a Gaussian $H_2$-valued random variable. Its covariance is given by
\begin{align*}
	\inpro[H_2]{\Cov(\tilde X) u }{v} &= \sum^\infty_{j=1} \inpro[H_2]{u}{e_j} \inpro[H_2]{v}{e_j}  \\
	&= \sum^\infty_{j=1} \lrinpro[\Cov(X)^{\frac{1}{2}}(H_1)]{I^*_{\Cov(X)^\frac{1}{2}(H_1) \hookrightarrow H_2} u}{e_j} \\ &\hspace{4em}\times \lrinpro[\Cov(X)^{\frac{1}{2}}(H_1)]{I^*_{\Cov(X)^\frac{1}{2}(H_1) \hookrightarrow H_2} v}{e_j} \\
	&= \lrinpro[H_2]{I_{\Cov(X)^\frac{1}{2}(H_1) \hookrightarrow H_2} I^*_{\Cov(X)^\frac{1}{2}(H_1) \hookrightarrow H_2}u}{v}
\end{align*}
for $u,v \in H_2$. It still holds that $\Cov(X) = I_{\Cov(X)^{1/2}(H_1) \hookrightarrow H_1} I^*_{\Cov(X)^{1/2}(H_1) \hookrightarrow H_1}$, whence 
\begin{equation}
	\label{eq:covariance-in-supspace}
	\begin{split}
		\Cov(\tilde X) &= I_{\Cov(\tilde X)^\frac{1}{2}(H_2) \hookrightarrow H_2} I^*_{\Cov(\tilde X)^\frac{1}{2}(H_2) \hookrightarrow H_2} \\ 
		&= I_{\Cov(X)^\frac{1}{2}(H_1) \hookrightarrow H_2} I^*_{\Cov(X)^\frac{1}{2}(H_1) \hookrightarrow H_2} = I_{H_1 \hookrightarrow H_2} \Cov(X) I^*_{H_1 \hookrightarrow H_2}.
	\end{split}
\end{equation}
In a certain sense, the distribution of $\tilde X$ is invariant with respect to the choice of the space $H_2$. Indeed, the reproducing kernel Hilbert space $\Cov(\tilde X)^{1/2}(H_2) = \Cov(X)^{1/2}(H_1)$ remains the same regardless of what space $H_2$ we use. 

We can expand $X$ as a cylindrical random variable in terms of $\tilde X$ in a converse version of~\eqref{eq:trace-version-of-cylindrical-rv}. This can be done using the fact that since $H_1 \hookrightarrow H_2$ densely and continuously, there is a linear, self-adjoint, densely defined and positive semidefinite operator $B$ on $H_2$ such that $\norm[H_1]{u} = \norm[H_2]{B^{1/2}u}$ for all $u \in \cD(B^{1/2})=H_1$. We summarize the argument for this from~\cite[Section~1.2]{LM72}. Let $\cD(B)$ be the set of all $u \in H_1$ such that $\inpro[H_1]{u}{\cdot}$ is continuous with respect to the $H_2$ norm. Then, since $H_1$ is dense in $H_2$, the identity $\inpro[H_1]{u}{v} = \inpro[H_2]{B u}{v}$ for all $v \in H_1$ defines an unbounded, self-adjoint and positive definite linear operator with domain $\cD(B)$ (so fractional powers are well-defined). To see that $\cD(B)$ is dense in $H_1$, note that by the Hahn--Banach theorem, this holds if and only if 
\begin{equation*}
	\inpro[H_1]{v}{y} = 0 \, \forall y \in \cD(B) \implies v = 0
\end{equation*}
for arbitrary $v \in H_1$. Let $v \in H_1$ be arbitrary and suppose that $\inpro[H_1]{v}{y} = 0 \, \forall y \in \cD(B)$. In particular, this holds for $y = I_{H_1 \hookrightarrow H_2}^* u$, $u \in H_2$, since $\inpro[H_1]{y}{\cdot} = \inpro[H_2]{u}{\cdot}$ is continuous on $H_2$. Therefore $\inpro[H_2]{v}{u} = 0$ for all $u \in H_2$, so $v = 0$. To see that $\cD(B^{1/2}) = H_1$, note that for $u \in \cD(B)$, $\norm[H_2]{B^{1/2} u }^2 = \inpro[H_2]{B u}{u} = \norm[H_1]{u}^2$. By density, this extends to $H_1$, which shows that $H_1 \subset \cD(B^{1/2})$. If $v \in \cD(B^{1/2})$ then 
\begin{align*}
	\norm[H_1]{v} = \hspace{-2pt} \sup_{\substack{u \in H_1 \\ \norm[H_1]{u} = 1}} |\inpro[H_1]{v}{u}| &\le \hspace{-2pt} \sup_{\substack{u \in \cD(B) \\ \norm[H_1]{u} = 1}} |\inpro[H_1]{v}{u}| \\ &= \hspace{-2pt} \sup_{\substack{u \in \cD(B) \\ \norm[H_1]{u} = 1}} |\inpro[H_2]{v}{B u}| = \hspace{-2pt} \sup_{\substack{u \in \cD(B) \\ \norm[H_1]{u} = 1}} |\inpro[H_2]{B^{1/2} v}{B^{1/2} u}| \le \norm[H_2]{B^{1/2} v},
\end{align*}
which shows that $\cD(B^{1/2}) \subset H_1$. Here we made use of the Cauchy--Schwarz inequality along with the identity $\norm[H_2]{B^{1/2} u } = \norm[H_1]{u}$ for $u \in H_1$, which we proved above. By density of $\cD(B)$ in $\cD(B^{1/2})$, we may pick an orthonormal basis $(f_j)_{j=1}^\infty$ of $H_1$ contained in $\cD(B)$. By definition of the covariance, we find that for $f_j \notin \ker(\Cov(X))$,
\begin{align*}
	&\E\Big[\Big|\inpro[H_1]{X}{f_j} - \sum_{k = 1}^N  \inpro[H_1]{X}{\Cov(X)^{-1} e_k} \inpro[H_2]{e_k}{B f_j}\Big|^2\Big] 
	\\
	&\,=\E[\inpro[H_1]{X}{f_j}^2] - 2 \sum_{k=1}^{N} \inpro[H_1]{f_j}{e_k} \inpro[H_2]{e_k}{B f_j} \\ &\qquad+ \sum_{k,\ell = 1}^{N} \inpro[H_1]{\Cov(X)^{-1/2} e_k}{\Cov(X)^{-1/2} e_\ell} \inpro[H_2]{e_k}{B f_j} \inpro[H_2]{e_\ell}{B f_j} \\
	&\,= \norm[H_1]{\Cov(X)^{1/2} f_j }^2 -2 \sum_{k=1}^{N} \inpro[H_2]{B f_j}{e_k} \inpro[H_1]{f_j}{e_k} + \sum^N_{k=1} \inpro[H_2]{B f_j}{e_k}^2 \\
	&\,= \norm[\Cov(X)^{1/2}(H_1)]{\Cov(X) f_j}^2 - 2 \sum_{k=1}^{N} \inpro[H_1]{f_j}{e_k}^2 + \sum_{k=1}^{N} \inpro[H_1]{f_j}{e_k}^2\\
	&\,= \norm[\Cov(X)^{1/2}(H_1)]{\Cov(X) f_j}^2 - \sum_{k = 1}^N \inpro[\Cov(X)^{1/2}(H_1)]{\Cov(X) f_j}{e_k}^2,
\end{align*}
so that, in light of~\eqref{eq:trace-version-of-cylindrical-rv}, $\inpro[H_1]{X}{f_j} = \inpro[H_2]{\tilde X}{B f_j}$ for $f_j \notin \ker(\Cov(X))$. This identity also holds true for $f_j \in \ker(\Cov(X))$, since then both sides are zero. Using this along with the continuity of $\inpro[H_1]{X}{\cdot} \colon H_1 \to L^2(\Omega,\R)$, it follows that
\begin{equation}
	\label{eq:cylindrical-rv-from-trace-version}
	\inpro[H_1]{X}{u} = \sum_{j = 1}^\infty \inpro[H_1]{X}{f_j} \inpro[H_1]{u}{f_j} = \sum_{j = 1}^\infty \inpro[H_2]{\tilde X}{B f_j} \inpro[H_1]{u}{f_j}
\end{equation}
for all $u \in H_1$, where the convergence takes place in $L^2(\Omega,\R)$. From here on, we make no notational distinction between $X$ and $\tilde X$.

A pair~$(X,Y)$ of random variables $X$ and $Y$ with values in two Hilbert spaces $H_1$ and $H_2$ is said to be jointly Gaussian if $X \oplus Y$ is an $H_1 \oplus H_2$-valued Gaussian random variable. Then $X$ and $Y$ are independent if and only if $\Cov(X,Y) =0$, cf.~\cite{M84}. We quote a theorem regarding the conditional distribution of $Y$ given $X$ from \cite{GM08}. To make sense of it, we first need a preliminary lemma \cite[Lemma~2.2]{GM08}. Since $\Cov(X)$ is of trace class, it is also compact and thus has an associated eigenbasis $(e_j)_{j=1}^\infty$ along with a sequence of eigenvalues $(\mu_j)_{j=1}^\infty$ which in this case is summable. 	
\begin{lemma}
	\label{lem:conditional-lemma}
	Let $X$ be an $H_1$-valued Gaussian random variable with zero mean, image measure $\mu = \IP \circ X^{-1} \colon H_1 \to \R$ and eigenpairs $(\mu_j, e_j)_{j=1}^\infty$. Let $\Gamma \in \cL_2(H_1,H_2)$. Then, the sum defining the operator
	\begin{equation*}
		\Gamma \Cov(X)^{-1/2} := \sum_{j=1}^{\infty} \frac{1}{\sqrt{\mu_j}} \inpro[H_1]{\cdot}{e_j} \Gamma e_j 
	\end{equation*}
	converges in $L^2((H_1,\mu),H_2)$ and
	\begin{equation*}
		%\label{eq:lem:conditional-lemma-1}		
		\norm[L^2((H_1,\mu),H_2)]{\Gamma \Cov(X)^{-1/2}}^2 := \int_{H_1} \norm[H_2]{\Gamma \Cov(X)^{-1/2} x}^2 \dd \mu(x) = \norm[\cL_2(H_1,H_2)]{\Gamma}^2.
	\end{equation*}
	Moreover, there exists a Borel subspace $\cM \subset H_1$ with $\mu(\cM) = 1$, the operator $\Gamma \Cov(X)^{-1/2}$ is linear on $\cM$ and 
	\begin{equation*}
		\Gamma \Cov(X)^{-1/2}x = \sum_{j=1}^{\infty} \frac{1}{\sqrt{\mu_j}} \inpro[H_1]{x}{e_j} \Gamma e_j
	\end{equation*}
	for all $x \in \cM$.
\end{lemma}

Note that in Lemma~\ref{lem:conditional-lemma}, it holds that $\Cov(X)^{1/2}(H_1) \subset \cM$, see, e.g., \cite[Lemma~3]{T79}. Moreover, $\Gamma \Cov(X)^{-1/2} \Cov(X)^{1/2} x = \Gamma x$ for $x \in H_1$. The mapping defined by
\begin{equation*}
	v \mapsto \sum_{j = 1}^\infty \frac{1}{\sqrt{\mu_j}}\inpro[H_1]{ X }{e_j} \inpro[H_1]{v}{e_j}
\end{equation*}
defines a standard cylindrical Gaussian random variable in $H_1$ with the sum converging $\IP$-a.s and in $L^2(\Omega,\R)$. Since $\Gamma \in \cL_2(H_1,H_2)$, the cylindrical Gaussian random variable in $H_2$ defined by 
\begin{equation*}
	u \mapsto \sum_{j = 1}^\infty \frac{1}{\sqrt{\mu_j}}\inpro[H_1]{ X }{e_j} \inpro[H_1]{\Gamma^* u}{e_j}
\end{equation*}
is induced by a Gaussian $H_2$-valued random variable, which can be taken to be the random variable $\Gamma \Cov(X)^{-1/2} X$ with $\Gamma \Cov(X)^{-1/2}$ defined as in Lemma~\ref{lem:conditional-lemma}. This operator show up in the next result, which is \cite[Theorem~2.4]{GM08}.

\begin{theorem}
	\label{thm:conditional-gaussian}
	Let $X$ and $Y$ be jointly Gaussian random variables with values in $H_1$ and $H_2$ such that $\Cov(X)$ is an injective operator. Then
	\begin{enumerate}[label=(\roman*)]
		\item $\Cov(X,Y)(H_2) \subset \Cov(X)^{1/2}(H_1)$ and $\Cov(X)^{-1/2} \Cov(X,Y) \in \cL_2(H_2,H_1)$,
		\item $\E[Y | X] = \E[Y] + (\Cov(X)^{-1/2} \Cov(X,Y))^* \Cov(X)^{-1/2} (X-\E[X])$,
		\item the conditional distribution of $Y$ given $X$ is Gaussian with (conditional) mean $\E[Y | X]$ and covariance $\Cov(Y) - (\Cov(X)^{-1/2} \Cov(X,Y))^* \Cov(X)^{-1/2} \Cov(X,Y)$.
	\end{enumerate}
\end{theorem}

We conclude this section by introducing a (strongly) cylindrical Wiener process in a Hilbert space $H$. It is, following~\cite{DPZ14, R11}, defined as a linear transformation $H \ni v \mapsto W_v$ whose values are real-valued Wiener processes with respect to $(\Omega, \cA, (\cF_t)_{t \in [0,T]}, \IP)$ having the covariance structure $\E[W_u(s) W_v(t)] = \min(s,t) \inpro[H]{Qu}{v}$ for $u,v \in H$, $s,t \in [0,T]$ and some $Q \in \Sigma^+(H)$. The operator $Q$ is referred to as the covariance operator of $W$. It\^o integrals taking values in $H$ in the form of
\begin{equation}
	\label{eq:stoch-int}
	\int_{0}^{t} \Psi(s) \dd W(s),
\end{equation}
$t \in [0,T]$, are well-defined for deterministic processes $\Psi \in L^2([0,T],\cL_2(Q^{1/2}(H),H))$ \cite[Section~4.2.1]{DPZ14}. In particular,~\eqref{eq:stoch-int} defines an $H$-valued Gaussian random variable with mean $0$ and covariance 
\begin{equation}
	\label{eq:stoch-int-cov}
	\Cov\left(\int_{0}^{t} \Psi(s) \dd W(s) \right) = \int_{0}^{t} \Psi(s) \Psi^*(s) \dd s,
\end{equation}
\cite[Proposition~4.28]{DPZ14}. By~\eqref{eq:schatten-bound-2}, the integrand takes value in a separable Banach space so the integral is well-defined in the Bochner sense.

\subsection{Functional analytic framework}
\label{subsec:fa-framework}

In this section, we introduce the functional analytic framework of \cite[Appendix~B]{K14} that we consider in the remainder of the paper. We fix a real, separable Hilbert space~$(H, \inpro{\cdot}{\cdot}, \norm{\cdot})$ and let $A: \cD(A) \subset H \rightarrow H$ be a densely defined linear operator, which is self-adjoint and positive definite with a compact inverse. By the spectral theorem, we obtain a sequence $(\lambda_j)_{j=1}^\infty$ of positive non-decreasing eigenvalues of $A$ with $\lim_{j \to \infty} \lambda_j = \infty$, along with an orthonormal eigenbasis $(e_j)_{j=1}^\infty$ in $H$. We define fractional powers of $A$ by 
\begin{equation*}
	%\label{eq:fracpower}
	A^{\frac{r}{2}} v := \sum_{j = 1}^{\infty} \lambda_j^{\frac{r}{2}} \inpro{v}{e_j}.
\end{equation*}
For $r \le 0$, this defines an operator in $\Sigma^+(H)$. For $r > 0$, $A^{\frac{r}{2}}$ is a densely defined, positive definite, self-adjoint, and unbounded operator with domain 
\begin{equation*}
	\dom(A^{\frac{r}{2}}) = \left\{ v \in H : \sum^\infty_{j=1} \lambda_j^r |\inpro{v}{e_j}|^2 < \infty \right\}.
\end{equation*}
We write $\dot{H}^r := \dom(A^{r/2})$ with $\dot{H}^0 = H$. This is a Hilbert space with respect to the inner product 
\begin{equation}
	\label{eq:fracpower-norm}
	\inpro[\dot{H}^r]{u}{v} = \inpro{A^{r/2}u}{A^{r/2}v} = \sum_{j = 1}^{\infty} \lambda_j^{r} \inpro{u}{e_j} \inpro{v}{e_j}.
\end{equation}
For $r<0$, we define $\dot{H}^r$ to be the completion of $H$ under the norm defined by~\eqref{eq:fracpower-norm}. The operator $A^{r/2}$ extends to a bounded operator on $\dot{H}^r$ and we may write
\begin{equation*}
	\dot{H}^r = \dom(A^{\frac{r}{2}}) = \left\{x = \sum_{j=1}^{\infty} x_j e_j : (x_j)_{j=1}^\infty \subset \R \text{ such that } \norm[\dot{H}^r]{x}^2 = \sum_{j=1}^{\infty} \lambda_j^{r} x_j^2 < \infty \right\}.
\end{equation*} 
Regardless of the sign of $r$, $\dot{H}^r$ is a Hilbert space with inner product $\inpro[\dot{H}^r]{u}{v} = \inpro{A^{r/2}u}{A^{r/2}v}$. % Dual characterization? 
For $r > s$, the embedding $\dot{H}^r \hookrightarrow \dot{H}^s$ is dense and compact. Moreover, Lemma~2.1 in~\cite{BKK20} allows us to, for all $r, s \in \R$, extend $A^{s/2}$ to an isometric isomorphism from $\dot{H}^r$ to $\dot{H}^{r-s}$, and we do so without changing notation. This means in particular that $\dot{H}^{r-s} = A^{s/2}(\dot{H}^r)$. Note also, that since $A^{r-s}$ is bounded on $\dot{H}^s$ when $s \ge r$, we have
\begin{equation}
	\label{eq:A-as-adjoint-embedding}
	\inpro[\dot{H}^s]{I_{\dot{H}^s \hookrightarrow \dot{H}^r}^* u}{v} = \inpro[\dot{H}^r]{u}{v} = \inpro[\dot{H}^s]{A^{\frac{r-s}{2}}u}{A^{\frac{r-s}{2}}v} = \inpro[\dot{H}^s]{A^{r-s}u}{v}
\end{equation}
for all $u,v \in \dot{H}^s$, so that by density, $I_{\dot{H}^s \hookrightarrow \dot{H}^r}^* = A^{r-s}$ on $H^r$.

For an operator $\Gamma \in \cL(H)$ and all $r \ge 0$, $\Gamma \in \cL(H,\dot{H}^r)$ if and only if $\Gamma^* \in \cL(H)$ can be continuously extended to $\dot{H}^{-r}$, and $\norm[\cL(H,\dot{H}^r)]{\Gamma} = \norm[\cL(\dot{H}^{-r},H)]{\Gamma^*}$. Moreover, $\Gamma \in \cL_2(H,\dot{H}^r)$ if and only if $\Gamma^* \in \cL_2(\dot{H}^{-r},H)$, and $\norm[\cL_2(H,\dot{H}^r)]{\Gamma} = \norm[\cL_2(\dot{H}^{-r},H)]{\Gamma^*}$. We assume that there is some $\zeta >0$ such that $I_{\dot{H}^\zeta \hookrightarrow H} \in \cL_2(\dot{H}^\zeta,H)$. We remark that then, for arbitrary $s \in \R$, 
\begin{equation}
	\label{eq:hs-invariance}
	\norm[\cL_2(\dot{H}^{s + \zeta}, \dot{H}^s)]{I_{\dot{H}^{s + \zeta} \hookrightarrow \dot{H}^s}}^2 = \sum_{j = 1}^{\infty} \norm[\dot{H}^s]{A^{-\frac{s + \zeta}{2}} e_j }^2 = \sum_{j = 1}^{\infty} \norm{A^{-\frac{\zeta}{2}}e_j}^2 = \norm[\cL_2(\dot{H}^{\zeta}, H)]{I_{\dot{H}^{\zeta} \hookrightarrow H}}^2 < \infty.
\end{equation}

The operator $-A$ is the generator of an analytic $C_0$-semigroup $S:=(S(t))_{t \ge 0}$ of bounded linear operators on~$H$ that extends to $\dot{H}^r$ for arbitrary $r<0$ and is a $C_0$-semigroup also on~$\dot{H}^r$, $r \in \R$. It has the spectral representation
\begin{equation*}
	S(t) v = \sum_{j=1}^\infty e^{-\lambda_j t} \inpro{v}{e_j} e_j
\end{equation*}
for $v \in H$, and we see that $S(t)(H) \subset \dot{H}^r$ for all $r \in \R$ and $t > 0$. In fact, for $r \ge 0$, there exists a constant $C < \infty$ such that for all $t > 0$ and $u \in H$
\begin{equation}
	\label{eq:semigroup-bound}
	\norm[\dot{H}^r]{S(t)u} = \norm{A^{\frac{r}{2}} S(t)u} \le C t^{-r/2} \norm{u}.
\end{equation}
Note that $S(t)$ commutes with $A^{r/2}$, for all $t \in [0,T]$ and $r \in \R$.  
We end this section with the introduction of the canonical example of $A$ as an elliptic operator.
\begin{example}
	\label{ex:elliptic-operator}
	Let $H = L^2(\cD)$, the space of square-integrable functions on a bounded domain $\cD \subset \R^d$, $d = 1,2,3$, which is either convex or has boundary $\partial \cD$ of class $\cC^2$. The differential operator
	\begin{equation*}
		\cA = -\sum^{d}_{i,j=1} \frac{\partial}{\partial \xi_i} a_{i,j}  \frac{\partial}{\partial \xi_j} + a_0
	\end{equation*}
	is equipped with homogeneous boundary conditions, of either Dirichlet or Neumann type. The coefficients $a_{i,j}$, $i,j = 1, \ldots, d$, are $\cC^1(\cD)$ functions  fulfilling $a_{i,j} = a_{j,i}$. Moreover, we assume that there is a constant $\lambda_0 > 0$ such that for all $y \in \R^d$ and almost all $\xi \in \cD$, $\sum^{d}_{i,j=1} a_{i,j} (\xi) y_i y_j \ge \lambda_0 |y|^2$. The function $a_0 \in L^\infty(\cD)$ is non-negative almost everywhere on $\cD$. We denote by $H^m(\cD) = W^{m,2}$ the classical Sobolev space of order $m \in \N$. Following \cite[Chapters~1-2]{Y10}, we let $a$ be a continuous, symmetric bilinear form
	\begin{equation*}
		a(u,v) = \sum_{i,j = 1}^d \int_\cD a_{i,j} \frac{\partial u}{\partial \xi_i} \frac{\partial v}{\partial \xi_j} \dd \xi + \int_\cD a_0 u v \dd \xi,
	\end{equation*}
	on  $V \subset H^1(\cD) \subset L^2(\cD)=H$. In the case of Dirichlet boundary conditions we take $V = H^1_0(\cD) = \{v \in H^1 : \gamma v = 0 \}$, where $\gamma$ is the trace operator. In the case of Neumann boundary conditions, we take $V = H^1$, and assume in addition that there is a constant $c_0 > 0$ such that $a_0 \ge c_0$ a.e.\ on $\cD$. Then $a$ is coercive and, with $V \subset H \subset V^*$ being a Gelfand triple, there exists a unique isomorphism $L \colon V \to V^*$ such that ${}_{V^*} \langle L u, v \rangle_V = a(u,v)$ for all $u, v \in V$. Viewing $L$ as an operator on $V^*$, it is densely defined and closed. We let $A = L|_{H}$ be its restriction to $H$ with domain $D(A) = \{v \in V : A v = L v \in H\}$. Then, $A$ is densely defined and positive definite. Since the embedding $V \hookrightarrow H$ is compact, its inverse is compact.
	
	We can relate the spaces $(\dot{H}^{s})_{s \in [0,2]}$ to the fractional Sobolev spaces $(H^{s})_{s \in [0,2]}$. In the case of Dirichlet boundary conditions, we have
	\begin{equation}
		\label{eq:sobolev_id_1}
		\dot{H}^s = 
		\begin{cases}
			H^s & \text{ if } s \in [0,1/2), \\
			\left\{u \in H^s : \gamma u  = 0 \right\} & \text{ if } s \in (1/2,3/2) \cup (3/2,2],
		\end{cases}	
	\end{equation}
	with norm equivalence. In the case of Neumann boundary conditions, 
	\begin{equation}
		\label{eq:sobolev_id_2}
		\dot{H}^s = 
		\begin{cases}
			H^s & \text{ if } s \in [0, 3/2), \\
			\left\{u \in H^s : {\partial u}/{\partial \nu_\Lambda} = 0 \right\} &\text{ if } s \in (3/2,2],
		\end{cases}	
	\end{equation}
	where $\partial v / \partial \nu_\Lambda = \sum_{i,j = 1}^d n_i a_{i,j} \gamma D^{j} v$,
	with $(n_1, \ldots, n_d)$ being the outward unit normal to $\partial \cD$. If $\partial \cD$ is of class $\cC^2$, then~\eqref{eq:sobolev_id_1} holds for $s=3/2$ and for certain special domains, $\dot{H}^{3/2}$ in~\eqref{eq:sobolev_id_2} can be characterized as those functions in $H^{3/2}$ that satisfy the boundary condition in a weak sense. For these facts and further details on domains of powers of elliptic operators, we refer to~\cite[Sections~16.4-6]{Y10}. Since the embedding $I_{H^s \hookrightarrow H} \in \cL_2(H^s,H)$ if and only if $s > d/2$ (see \cite[p. 286]{T67}, \cite[Theorem~3.3.4(ii)]{ET96} for the case that $\partial \cD \in \cC^\infty$ but note that the result holds also for our case, cf.\ \cite[Lemma~2.3]{KLP22}), the identities~\ref{eq:sobolev_id_1} and~\ref{eq:sobolev_id_2} imply that we may choose $\zeta > d/2$ in this setting.
\end{example}
	
\section{SPDE bridges with observation noise}
\label{sec:spde-bridges}

In this section, we introduce the linear SPDE that we work with throughout the paper, along with an assumption ensuring continuity of the mild solution~$X \colon \Omega \times [0,T] \to H$. We then define an SPDE bridge $X^{x,y}$, which is the process $X$ fulfilling $X(0) = x$ conditioned, in a certain sense, on $X^x(T) + Z = y$ for a cylindrical Gaussian random variable $Z$.

With $A$ fulfilling the assumptions of Section~\ref{subsec:fa-framework}, we consider SPDEs of the form 
\begin{equation}
	\label{eq:SPDE}
	\begin{split}
		\dd X^x(t) + A X^x(t) \dd t &= \dd W(s) \text{ for } t \in (0,T]\\
		X^x(0) & = x,
	\end{split}
\end{equation}
for $T <\infty$. Here, the initial value $x \in H$ is taken to be deterministic and $W$ is a cylindrical Wiener process in $H$ with covariance operator $Q \in \Sigma^+(H)$. A stochastic process $X^x \in \cC([0,T],L^2(\Omega,H))$ is said to be a mild solution of~\eqref{eq:SPDE} if 
\begin{equation*}
	%\label{eq:mild-solution}
	X^x(t) = S(t) x + \int^t_0 S(t-s) \dd W(s)
\end{equation*}
for all $t \in [0,T]$. Recall that $S$ is the semigroup generated  by $-A$. For this solution to exist, we need the following assumption.
\begin{assumption}
	\label{ass:Q}
	There is a parameter $\beta > 0$ such that 
	\begin{equation*}
		\norm[\cL_2(Q^{\frac{1}{2}}(H),\dot{H}^{\beta-1})]{I_{Q^{\frac{1}{2}}(H)\hookrightarrow\dot{H}^{\beta-1}}} = \norm[\cL_2(H)]{A^{\frac{\beta-1}{2}} Q^{\frac{1}{2}}} < \infty.
	\end{equation*}
\end{assumption}

\begin{example}
	\label{ex:Q-reg}
	Let us consider this assumption in the context of Example~\ref{ex:elliptic-operator}. For $Q = I$, the process $W$ is usually referred to as space-time white noise. We obtain from~\eqref{eq:hs-invariance} that Assumption~\ref{ass:Q} is fulfilled for all $\beta \le 1 - \zeta$, i.e., we must have $d = 1$, and we may choose $\beta < 1/2$. Another popular choice is to consider noise which is white in time and correlated, but spatially homogeneous, in $\cD$. In this case $Q$ is given by
	\begin{equation*}
		Q u (\xi) = \int_{\cD} q(\xi-\upsilon) u(\upsilon) \dd \upsilon.
	\end{equation*} 
	Here, $q \colon \R^d \to \R$ is a positive definite, symmetric, continuous and integrable function. If this function is sufficiently smooth, $W(1)$ is a Gaussian random field on $\cD$ and $q(\xi-\upsilon) = \Cov(W(1,\xi),W(1,\upsilon))$ for $\xi,\upsilon \in \cD$. If we assume that the Fourier transform $\hat{q}$ of $q$ fulfills $\hat{q}(\xi) \le C \left(1 + |\xi|^2 \right)^{-\sigma}$ for some $\sigma > d/2$ and all $\xi \in \R^d$, then Assumption~\ref{ass:Q} is fulfilled for all $\beta < 1 + \min(\sigma-d/2,1/2)$ in the case of Dirichlet boundary conditions and for all $\beta < 1 + \min(\sigma-d/2,3/2)$ in the case of Neumann boundary conditions \cite[Corollary~4.4]{KLP22}.
\end{example}

In light of~\eqref{eq:semigroup-bound}, it follows from Assumption~\ref{ass:Q} that for all $r < \beta$ and $\epsilon \in (0,\min(\beta-r,1))$,
\begin{equation*}
	\label{eq:Q:cont}
	\int^T_0 t^{-\epsilon} \norm[\cL_2(Q^{1/2}(H),\dot{H}^r)]{S(t)}^2 \dd t = \int^T_0 t^{-\epsilon} \norm[\cL_2(H)]{A^{\frac{r}{2}} S(t)Q^{1/2}}^2 \dd t  < \infty.
\end{equation*}
Using this bound, we obtain existence and uniqueness of the solution $X^x$ \cite[Chapters~5-6]{DPZ14}, as well the existence of a continuous modification of this process \cite[Theorem~5.11]{DPZ14}. This means that we may regard $X^x = X^x(\cdot)$ as a random variable in $\cC_T := \cC([0,T],H)$ as well as $H_T := L^2([0,T],H)$ \cite[Proposition~3.18]{DPZ14}. When $x\in \dot{H}^r$, the same statements hold with $H$ replaced by $\dot H^r$, for arbitrary $r \in [0,\beta)$. In particular,
\begin{equation*}
	%\label{eq:mild-solution-reg}
	\sup_{t \in [0,T]}
	\E \left[\norm[\dot{H}^r]{X^0(t)}^p\right]
	\le \E \left[\sup_{t \in [0,T]} \norm[\dot{H}^r]{X^0(t)}^p\right] < \infty
\end{equation*}
for all $t \in (0,T]$, $p \ge 1$. As a consequence of the construction of the stochastic integral, the law of $X^x(t)$ is Gaussian for each $t \in [0,T]$ \cite[Theorem~5.2]{DPZ14}. It is therefore determined by the mean of $X^x(t)$ and its covariance, which in light of~\eqref{eq:stoch-int-cov} and the fact that $S$ takes values in $\Sigma^+(H)$ is given by
\begin{equation*}
	%\label{eq:covXt}
	Q(t) := \Cov(X^x(t)) = \Cov(X^0(t)) = \int^t_0 S(t-s) Q S(t-s) \dd s = \int^t_0 S(s) Q S(s) \dd s.
\end{equation*}
As a consequence of $X^0(t)$, $t \in [0,T]$, being $\dot{H}^r$-valued, $Q(t)$ extends to $\dot{H}^{-r}$ with 
\begin{equation*}
	%\label{eq:covXt-reg}
	\sup_{t \in [0,T]} \norm[\cL_1(H)]{A^{\frac{r}{2}}Q(t)A^{\frac{r}{2}}} = \sup_{t \in [0,T]} \norm[\cL_1(\dot{H}^{-r},\dot{H}^{r})]{Q(t)} < \infty
\end{equation*}
for all $r < \beta$.
By \cite[Theorem~5.2]{DPZ14}, the law of $X^x$ in $H_T$ is also Gaussian with mean $S(\cdot) x$ and covariance $\bar Q \in \cL_1(H_T)$ given by 
\begin{equation}
	\label{eq:covX-HT}
	\bar Q u := \Cov(X^x) u = \Cov(X^0) u= \int^T_0 \left(\int^{\min(t,\cdot)}_0 S(\cdot - s) Q S(t - s) \dd s\right) u(t) \dd t
\end{equation}
for $u \in H_T$. Since $\cC_T$ is dense in $H_T$, $H_T = H_T^*$ is weak-* dense in $\cC_T^*$. Thus, there is a sequence $(u_j)_{j=1}^\infty \subset H_T^*$ such that $\dualp[\cC_T]{X}{u} = \lim_{j \to \infty} \inpro[H_T]{X}{u_j}$ $\IP$-a.s. Since the almost sure limit of Gaussian random variables is Gaussian, it follows that the law of $X^x$ is Gaussian in $\cC_T$, too. 
The pair $(X^x, X^x(T))$, with $X^x\in L^2(\Omega,H_T)$, $X^x(T)~\in L^2(\Omega,H)$, is jointly Gaussian \cite{GM08}. 	
This fact can, by~Theorem~\ref{thm:conditional-gaussian}, be used to define a bridge process by conditioning $X^x$ on the observation $X^x(T) = y$, in a certain sense. This is the approach taken in \cite{GM06} and \cite{GM08}, where the authors define $X^{x,y}(t)$, for an appropriate $y \in H$ and $t \in [0,T)$, by the formula
\begin{equation*}
	X^{x,y}(t) = X^x(t) - \left(\Cov(X^0(T))^{-\frac{1}{2}} \Cov(X^0(T),X^x(t))\right)^* \Cov(X^0(T))^{-\frac{1}{2}} (X^x(T) - y).
\end{equation*}
Theorem 2.14 in~\cite{GM08} provides a justification for this formula, see also Proposition~\ref{prop:bridge-conditional-property} below.
For this to be well-defined, $\Cov(X^0(T))$ has to be injective, which is the reason for why $Q$ is assumed to be injective in \cite{GM06,GM08}. Note that if $H$ is finite-dimensional, $\Cov(X^0(T))$ is invertible and the formula reduces (without any restrictions on $y$) to the familiar expression
\begin{equation*}
	X^{x,y}(t) = X^x(t) - \Cov(X^x(t),X^0(T))\Cov(X^0(T))^{-1} (X^x(T) - y).
\end{equation*}

In many applications of stochastic spatio-temporal processes, which $X^x$ is a model of, the process can not be observed exactly. This motivates our generalization of the results of~\cite{GM08}, where we condition on $X^x(T) + Z = y$, with $Z$ being a cylindrical Gaussian random variable modeling observation noise. Under the following assumption, $\Cov(X^0(T) + Z)$ becomes an injective operator without any extra assumption on $Q$, so that we may treat a larger class of SPDEs. Moreover, as is seen in the next sections, the addition of $Z$ allows us to derive convergence rates for a spatially semidiscrete approximation of $X^{x,y}$. 

\begin{assumption}
	\label{ass:Z}
	The observation noise $Z$ is a Gaussian cylindrical random variable in $H$, such that
	\begin{enumerate}[label=(\roman*)]
		\item its covariance is given by $\Cov(Z) =: \tilde Q$,
		\item it is independent of $W$, 
		\item there is some $\eta \in (-\infty,\beta)$ such that $Z$ can be regarded as a Gaussian random variable in $\dot{H}^{\eta}$ and 
		\item \label{ass:Z:alpha} the embedding $\dot{H}^{\alpha} \hookrightarrow \tilde Q^{1/2}(H)$ holds true for some $\alpha \ge \max(\eta,0)$.
	\end{enumerate}
\end{assumption}
Since the embedding $I_{H \hookrightarrow \dot{H}^{-\zeta}}$ is Hilbert--Schmidt, the third assumption at least holds for $\eta = -\zeta$. In the case that $\eta \ge 0$, the observation noise $Z$ is a Gaussian $H$-valued random variable and its covariance in $\dot{H}^{\eta}$ is by~\eqref{eq:covariance-in-subspace} and~\eqref{eq:A-as-adjoint-embedding} given by $\Cov(Z) (I^*_{\dot{H}^{\eta} \hookrightarrow H})^{-1} = \tilde Q A^{\eta} \in \cL_1(\dot{H}^{\eta})$. Again we make no difference in notation between $\tilde Q$ and its extension to $\dot{H}^{-\eta}$. In the case that $\eta \le 0$, the same expression for the covariance is obtained from~\eqref{eq:covariance-in-supspace}. 

We write $\mu_T$ for the Gaussian image measure of $X^x(T) + Z$ on $\dot{H}^{\eta}$ when $x = 0$. This sum is interpreted as an $\dot{H}^{\eta}$-valued random variable with covariance 
\begin{equation}
	\label{eq:observation-w-noise-cov}
	\Cov(X^0(T) + Z) = (Q(T) + \tilde Q) A^{\eta} \in \cL_1(\dot{H}^{\eta}) \cap \Sigma^+(\dot{H}^\eta).
\end{equation}
This expression follows from the independence of $Z$ and $X$, in turn a consequence of Assumption~\ref{ass:Z} and the construction of the It\^o integral. The reproducing kernel Hilbert space $\Cov(X^0(T)+Z)^{1/2}(\dot{H}^{\eta})$ is dense in $\dot{H}^{\eta}$. To see this, it suffices to show that $\dot{H}^{\alpha} \hookrightarrow ((Q(T) + \tilde Q) A^{\eta})^{1/2}(\dot{H}^{\eta})$. This in turn follows from Assumption~\ref{ass:Z}, Proposition~\ref{app:inverse-bound} and the fact that $A^{\frac{\eta-\alpha}{2}}(\dot{H}^{\eta}) = \dot{H}^{\alpha}$ (see Section~\ref{subsec:fa-framework}) since for all $v \in \dot{H}^{2\max(\eta,0)}$,
\begin{equation}
	\label{eq:sum-covariance-bounded-below}
	%\begin{split}
	\norm[\dot{H}^{\eta}]{((Q(T) + \tilde Q) A^{\eta})^{\frac{1}{2}}v}^2 = \inpro[\dot{H}^{\eta}]{Q(T) A^{\eta} v}{v} +  \inpro[\dot{H}^{\eta}]{\tilde Q A^{\eta}v}{v} \ge \norm{{\tilde Q}^{\frac{1}{2}}A^{\eta}v}^2 \gtrsim \norm[\dot{H}^{\eta}]{A^{\frac{\eta-\alpha}{2}}v}^2.
	%\end{split}
\end{equation}
In the first equality we used the fact that $((Q(T) + \tilde Q) A^{\eta})^{\frac{1}{2}}$ is symmetric on $\dot{H}^\eta$, a consequence of~\eqref{eq:observation-w-noise-cov}.
By density, this extends to $v \in \dot{H}^{\eta}$. This also shows injectivity of $\Cov(X^0(T)+Z)$.

Next, we introduce the operators that are used to define our bridge process $X^{x,y}$. First, we note that 
\begin{equation}
	\label{eq:crosscov-X-XT}
	\Cov(X,X^0(T)+Z) = Q(\cdot)S(T-\cdot) A^{\eta} \in \cL_1(\dot{H}^{\eta},H_T),
\end{equation}
where the operator $Q(\cdot)S(T-\cdot)$ is defined by $(Q(\cdot)S(T-\cdot) u) (t) = Q(t)S(T-t) u$ for $u \in H$, $t \in [0,T]$. In light of the definition of $\Cov(X,X^0(T)+Z)$ and the independence of $Z$ and $X$, \eqref{eq:crosscov-X-XT} follows (for $\eta \le 0$) from the fact that $\E\left[\inpro[\dot{H}^\eta]{X(T)}{u} \inpro[H_T]{X}{v}\right]$ is equal to
\begin{align*}
	&\E\left[\lrinpro[\dot{H}^\eta]{\int^T_0 S(T-r) \dd W(r)}{u} \left(\int^T_0 \lrinpro{\int^t_0 S(t-r) \dd W(r)}{v(t)} \dd t \right) \right] \\
	&\quad=\int^T_0 \E\left[ \lrinpro{\int^t_0 A^{\frac{\eta}{2}} S(T-r) \dd W(r)}{A^{\frac{\eta}{2}} u} \lrinpro{\int^t_0 S(t-r) \dd W(r)}{v(t)} \right] \dd t \\
	&\quad= \int^T_0 \lrinpro{ \left(\int^t_0 S(t-r) (A^{\frac{\eta}{2}} S(T-r) )^* \dd r\right) A^{\frac{\eta}{2}} u}{v(t)} \dd t \\
	&\quad= \int^T_0 \inpro{Q(t) S(T-t) A^\eta u}{v(t)} \dd t = \inpro[H_T]{Q(\cdot)S(T-\cdot) A^{\eta} u }{v}
\end{align*} 
for arbitrary $u \in \dot{H}^\eta$ and $v \in H_T$. Here, the second inequality follows from the independent increment property of $W$, the third from \eqref{eq:stoch-int-cov} and the polarization identity, and the fourth from the semigroup property of $S$. For $\eta > 0$ a density argument can be used. Similarly, $\Cov(X(t),X^0(T)+Z) = Q(t)S(T-t)A^{\eta}$. For $t \in [0,T]$, we write
\begin{align*}
	K(t) :=& \left(\Cov(X^0(T)+Z)^{-\frac{1}{2}} \Cov(X^0(T)+Z,X(t))\right)^* 
	\\ =& \left( ((Q(T) + \tilde Q) A^{\eta})^{-\frac{1}{2}} S(T-t)Q(t) \right)^* \in \cL_2(\dot{H}^{\eta},H).
\end{align*}
We write $\cK \colon \dot{H}^{\eta} \to \cC_T \subset H_T$ for the operator in $\cL_2(\dot{H}^{\eta},H_T)$ defined by
\begin{equation}
	\label{eq:kappa}
	\cK := \left(\Cov(X^0(T)+Z)^{-\frac{1}{2}} \Cov(X^0(T)+Z,X)\right)^*.
\end{equation}
These operators are well-defined as a consequence of Theorem~\ref{thm:conditional-gaussian}. Note that for $u \in \Cov(X^0(T)+Z)^{\frac{1}{2}}(\dot{H}^{\eta})$,
\begin{equation}
	\label{eq:K-Kt}
	\begin{split}
		\inpro[H_T]{\cK u}{v} &= \inpro[H_T]{Q(\cdot)S(T-\cdot) A^{\eta} ((Q(T) + \tilde Q) A^{\eta})^{-\frac{1}{2}} u}{v} \\
		&= \int^T_0 \inpro{Q(t)S(T-t) A^{\eta} ((Q(T) + \tilde Q) A^{\eta})^{-\frac{1}{2}} u}{v(t)} = \int^T_0 \inpro{K(t) u}{v(t)} \dd t
	\end{split}
\end{equation}
for all $v \in H_T$. By density, this holds for all $u \in \dot{H}^{\eta}$ so that $(\cK u)(t) = K(t) u$ for almost every $t \in [0,T]$. This extends to all $t \in [0,T]$ since $\cK$ maps into $\cC_T$. In fact, the stronger claim that $\cK$ is $\gamma$-radonifying holds true, which we formulate as a separate lemma. This technical result allows us to show that the process $X^{x,y}$ is well-defined both as an element of $H_T$ and as a continuous process with values in $H$. The proof is similar to that of~\cite[Lemma~3.2]{GM06}, but the inclusion of the cylindrical random variable $Z$ warrants a separate treatment.

\begin{lemma}
	\label{lem:K-is-radonifying}
	Under Assumptions~\ref{ass:Q} and~\ref{ass:Z}, the operator $\cK \colon \dot{H}^{\eta} \to \cC_T$ is $\gamma$-radonifying.
\end{lemma}
\begin{proof}
	First we note that for any $v \in \dot{H}^{\eta}$
	\begin{equation*}
		\norm[\dot{H}^{\eta}]{((Q(T) + \tilde Q) A^{\eta})^{\frac{1}{2}} v}^2 = \norm[\dot{H}^{\eta}]{(Q(T)A^{\eta})^{\frac{1}{2}} v}^2 + \norm[\dot{H}^{\eta}]{(\tilde QA^{\eta})^{\frac{1}{2}} v}^2. 
	\end{equation*}
	Moreover, for $v \in \dot{H}^{2\max(\eta,0)}$
	\begin{equation*}
		\norm[\dot{H}^{\eta}]{(Q(T)A^{\eta})^{\frac{1}{2}} v}^2 = \inpro{Q(T)A^{\eta}v}{A^{\eta}v} = \int^T_0 \norm{Q^{\frac{1}{2}} S(T-t)A^{\eta} v}^2 \dd t. 
	\end{equation*}
	By Assumption~\ref{ass:Q} and the density of $\dot{H}^{2\max(\eta,0)} \hookrightarrow \dot{H}^{\eta}$, this identity extends to $v \in \dot{H}^{\eta}$. Hence, for $u \in \Cov(X^0(T)+Z)^{\frac{1}{2}}(\dot{H}^{\eta}) = ((Q(T) + \tilde Q) A^{\eta})^{1/2}(\dot{H}^{\eta})$ and $v = ((Q(T) + \tilde Q) A^{\eta})^{-1/2} u \in \dot{H}^{\eta}$, we obtain from the previous two identities that 
	\begin{equation*}
		\int^T_0 \norm{Q^{\frac{1}{2}} S(T-t)A^{\eta} ((Q(T) + \tilde Q) A^{\eta})^{-1/2} u}^2 \dd t \le \norm[\dot{H}^{\eta}]{u}^2.
	\end{equation*}
	This inequality is true also for $u \in \dot{H}^{\eta}$, by density of $((Q(T) + \tilde Q) A^{\eta})^{1/2}(\dot{H}^{\eta}) \hookrightarrow \dot{H}^{\eta}$.
	
	Summing up, what we have shown above is that the operator $\cJ^0$ given by
	\begin{equation*}
		(\cJ^0 v) (t) := Q^{\frac{1}{2}} S(T-t)A^{\eta} ((Q(T) + \tilde Q) A^{\eta})^{-1/2} v
	\end{equation*}
	for $v \in ((Q(T) + \tilde Q) A^{\eta})^{1/2}(\dot{H}^{\eta})$ and almost every $t \in [0,T]$ is a well-defined contraction and extends to $\cL(\dot{H}^{\eta},H_T)$. The operator $\cK: \dot{H}^{\eta} \to H_T$ can be factored into $\cK = \cJ^1 \cJ^0$ where $\cJ^1 \colon H_T \to \cC_T \hookrightarrow H_T$ is given by
	\begin{equation*}
		(\cJ^1 v)(t) := \int^t_0 S(t-s) Q^{\frac{1}{2}} v(s) \dd s
	\end{equation*}
	for $v \in H_T, t \in [0,T]$. To see this, let $v \in \dot{H}^{\eta}$ be such that $u = ((Q(T) + \tilde Q) A^{\eta})^{1/2} v$ and note that
	\begin{align*}
		(\cJ^1 \cJ^0 u)(t) = \int^t_0 S(t-s) Q^{\frac{1}{2}} (\cJ^0 u)(s) \dd s
		&= \int^t_0 S(t-s) Q S(T-s)A^{\eta} v \dd s \\ &= \int^t_0 S(t-s) Q S(t-s) \dd s \, S(T-t)A^{\eta} v \\
		&= Q(t)S(T-t) A^{\eta} ((Q(T) + \tilde Q) A^{\eta})^{-\frac{1}{2}} u = K(t) u.
	\end{align*}
	Thus, in light of~\eqref{eq:K-Kt}, $\cK u = \cJ_1 \cJ_0 u$ for $u \in ((Q(T) + \tilde Q) A^{\eta})^{1/2}(\dot{H}^{\eta})$ so by density, $\cK = \cJ^1 \cJ^0$.
	
	By~\cite[Corollary~B.5]{DPZ14} and~\eqref{eq:covX-HT}, $\cJ^1(H_T) = {\bar Q}^{1/2}(H_T)$, the reproducing kernel Hilbert space of $X^x$ on $H_T$. By~\cite[Propositions~1.7, 2.10]{DPZ14}, this is also the reproducing kernel Hilbert space of $X^x$ on $\cC_T$, whence $I_{{\bar Q}^{1/2}(H_T) \hookrightarrow \cC_T}$ is $\gamma$-radonifying \cite[Proposition~8.6]{VN10}. Since $\cJ^1 \in \cL(H_T,{\bar Q}^{1/2}(H_T))$, $\cK = \cJ^1 \cJ^0 = I_{{\bar Q}^{1/2}(H_T) \hookrightarrow \cC_T} \cJ^1 \cJ^0$ is $\gamma$-radonifying by the ideal property of these operators \cite[Theorem~6.2]{VN10}.
\end{proof}

We are now ready to introduce the SPDE bridge $X^{x,y}$, the approximation of which is the main topic of this paper. This we do in a separate proposition, in which we show that this process has a continuous modification and therefore can be seen as a $\cC_T$-valued random variable.

\begin{proposition}
	\label{prop:bridge-existence}
	Under Assumption~\ref{ass:Q} and~\ref{ass:Z}, there exists a Borel subspace $\cM \subset \dot{H}^{\eta}$ such that $\mu_T(\cM) = 1$ and for all $x \in H$, $y \in \cM$, the $H$-valued Gaussian process 
	\begin{equation}
		\label{eq:prop:bridge-existence:existence-1}
		X^{x,y}(t) := X^x(t) - K(t) ((Q(T) + \tilde Q) A^{\eta})^{-\frac{1}{2}} (X^x(T) + Z - y),
	\end{equation}
	is well-defined for all $t \in [0,T]$ and has a continuous modification. Moreover, the representation
	\begin{equation}
		\label{eq:prop:bridge-existence:existence-2}
		X^{x,y}(t) = X^x(t) + K(t) ((Q(T) + \tilde Q) A^{\eta})^{-\frac{1}{2}} (y - S(T)x) - \hat{X}^0(t)
	\end{equation}
	holds true for all $t \in [0,T]$, $\IP$-a.s. Here 
	\begin{equation}
		\label{eq:prop:bridge-existence:existence-3}
		\hat{X}^0(t) := \E[X^0(t) | X^0(T) + Z] = K(t) ((Q(T) + \tilde Q) A^{\eta})^{-\frac{1}{2}} (X^0(T) + Z).
	\end{equation}
	for all $t \in [0,T]$.
\end{proposition}

\begin{proof}
	We first note that the equality in~\eqref{eq:prop:bridge-existence:existence-3} is a straightforward consequence of Theorem~\ref{thm:conditional-gaussian}. Next, by Lemma~\ref{lem:conditional-lemma} there is a subspace $\cM \subset \dot{H}^\eta$ with $\mu_T(\cM) = 1$ on which $\cK ((Q(T)+\tilde Q)A^{\eta})^{-1/2}$ is well-defined and linear. To obtain a continuous process, we explicitly choose $\cM$ by 
	\begin{equation*}
		\cM = \left\{ y \in \dot{H}^\eta : \sum_{j = 1}^\infty \frac{1}{\sqrt{\mu_{j}}} \inpro[\dot{H}^\eta]{y}{f_{j}} \cK f_{j} \text{ converges in } \cC_T \right\}
	\end{equation*}
	where $(\mu_{j},f_j)_{j=1}^\infty$ are the eigenpairs of $(Q(T)+\tilde Q)A^{\eta}$. Note that this is indeed a Borel subspace of $\dot{H}^\eta$ on which $\cK ((Q(T)+\tilde Q)A^{\eta})^{-1/2}$ is linear. That $\mu_T(\cM) = 1$ is equivalent to the $\IP$-a.s.\ convergence in $\cC_T$ of the sum
	\begin{equation*}
		\sum_{j = 1}^\infty \frac{1}{\sqrt{\mu_{j}}} \inpro[\dot{H}^\eta]{X(T) + Z}{f_{j}} \cK f_{j},
	\end{equation*}
	which holds as a consequence of $\cK$ being $\gamma$-radonifying, since $\IP$-almost sure and mean square convergence of series of independent mean zero Gaussian random variables in a Banach space are equivalent \cite[Corollary~6.4.4]{HvNVW17b}. Moreover, as in~\cite[Section~4]{GM06}, that $\cK$ is $\gamma$-radonifying implies the existence of a continuous modification of $\hat{X}^0$. The proof is completed by noting that $S(t)(H) \subset \cM$ for all $t > 0$. This follows from the fact that $S(t)(H) \subset \dot{H}^r$ for $t >0$ and all $r \in \R$ along with the embedding $\dot{H}^{\alpha} \hookrightarrow ((Q(T) + \tilde Q) A^{\eta})^{\frac{1}{2}}(\dot{H}^{\eta})$.
\end{proof}

The next proposition justifies why we refer to $X^{x,y}$ as $X^x$ conditioned on $X^x(T) + Z = y$. Its proof is based on the second part of Theorem~\ref{thm:conditional-gaussian} and is essentially a word-by-word repetition of that of~\cite[Theorem~2.14]{GM08}. We therefore omit it. 

\begin{proposition}
	\label{prop:bridge-conditional-property}
	Let $\Phi \colon L^2([0,T],H) \to \R$ be an arbitrary functional for which $\E[|\Phi(X^x)|] < \infty$. Under the same assumptions as Proposition~\ref{prop:bridge-existence},
	\begin{equation*}
		\E\left[\Phi(X^x) | X^x(T)+Z = y \right] = \E[\Phi(X^{x,y})]
	\end{equation*}
	for $\mu_T$-a.e. $y \in \dot{H}^{\eta}$. The left hand side of this equation is defined as a function $g_\Phi \in L^1((\dot{H}^{\eta},\mu_T),\R)$ of $y$ such that $g_\Phi(X^x(T)+Z) = \E\left[\Phi(X^x) | X^x(T)+Z \right]$ $\IP$-a.s.
\end{proposition}

\begin{remark}
	The proposition above is key to the approximation of the linear SPDE solution $X^x$ conditioned on $X^x(T) + Z = y$. This is accomplished via the approximation of the bridge process $X^{x,y}$ given by~\eqref{eq:prop:bridge-existence:existence-1} by spectral or finite element methods in the next parts of the paper. There is a large literature on such spatial approximations of SPDEs with nonlinearities (see, e.g., \cite{K14,AL16,AKL16,CH19,KLP20,CHS21} for some examples as well as \cite{JK09,LPS14} for surveys of earlier literature). However, our approach cannot directly extend these results to the approximation of nonlinear SPDE solutions $X^x$ conditioned on $X^x(T) + Z = y$. The reason for this is that the proof of Proposition~\ref{prop:bridge-conditional-property} relies on the process $X^x$ being Gaussian, as does the analysis of approximations of $X^{x,y}$ below. This only holds for linear SPDEs with additive noise. 
	
	However, in certain applications, our results might still be of interest to the nonlinear setting. In finite dimensions, a linear SDE bridge process can be used in parameter estimation of a nonlinear SDE model in a Markov chain Monte Carlo algorithm. This is in part accomplished by showing equivalence of the laws of the linear and nonlinear SDE bridge processes and approximating (a value proportional to) the Radon--Nikodym derivative \cite{SMZ17}. We are not aware of analogous results in the infinite-dimensional SPDE setting. However, in \cite{GM06} conditions were derived on a nonlinear stochastic reaction-diffusion equation such that its law at a fixed time is equivalent to that of the corresponding linear equation, see \cite[Proposition~5.1, Example~9.2]{GM06} for details. Moreover, the process $X^{x,y}$ (with $Z=0$) was used in \cite[Theorem~5.2]{GM06} to derive an expression for the corresponding Radon--Nikodym derivative. This is a tentative suggestion that, as in the  finite-dimensional setting, the approximation of linear SPDE bridge processes is relevant also for parameter estimation in nonlinear SPDE models.
\end{remark}

\section{Spatial approximation of SPDE bridges with observation noise}
\label{sec:approximation}

In this section, we introduce an abstract, spatially semidiscrete approximation $X^{x,y}_V$ of the SPDE bridge $X^{x,y}$. We consider a sequence $(V_i)_{i \in \cI}$ of finite-dimensional Hilbert subspaces of $H$, all equipped with the inner product of $H$. Here $\cI$ is a general index set. We fix a subspace $V \in (V_i)_{i \in \cI}$ and seek an approximation of the bridge process $X^{x,y}$ with values in $V$. We derive three technical lemmas related to this process. These are used in subsequent sections, in which we specify $V$ to correspond to well-known spatial discretization methods: the spectral method and the finite element method.

Since our observation noise $Z$ is cylindrical, we need to make sense of projections of cylindrical random variables onto $V$. This is covered by the following lemma. We write $P_V \colon H \to V$ for the orthogonal projection onto $V$ but in the lemma, $P_V Y$ is only formally the projection of a cylindrical random variable $Y$ in $H$ onto $V$. Recall that by the notation $\inpro{Y}{v}$ we refer to the evaluation of $Y$ on $v \in H$, see Remark~\ref{rem:cyl-notation}.

\begin{lemma}
	\label{lem:fin-dim-proj-well-def}
	Let $V$ be a finite-dimensional subspace of $H$ equipped with the inner product of $H$ and let $Y$ be a cylindrical Gaussian random variable on $H$. Then, the $V$-valued Gaussian random variable $P_V Y$ given by
	\begin{equation*}
		%\label{eq:cylindrical-projection}
		\inpro{P_V Y}{v} := \inpro{Y}{v}
	\end{equation*}
	for all $v \in V$
	is well-defined and $\sigma(P_V Y) \subset \sigma(Y)$.
\end{lemma}
\begin{proof}
	Let $(\phi_j)_{j=1}^N$ be an orthogonal basis of $V$ with $\dim(V) = N \in \N$. Setting
	\begin{equation*}
		P_V Y := \sum_{j=1}^N \inpro{Y}{\phi_j} \phi_j
	\end{equation*}
	defines a $V$-valued random variable which for a given $v \in V$, fulfills $\inpro{P_V Y}{v} = \inpro{Y}{v}$ by linearity of $\inpro{Y}{\cdot}$. To see that it is unique, suppose that $\tilde P_V Y$ is another random variable in $L^2(\Omega,V)$ such that $\inpro{\tilde P_V Y}{v} = \inpro{Y}{v}$ for all $v \in V$. Then 
	\begin{equation*}
		\E[\norm{P_V Y-\tilde P_V Y}^2] = \E\Big[\Big(\sup_{\substack{v \in V \\ \norm{v}=1}} |\inpro{P_V Y-\tilde P_V Y}{v}| \Big)^2\Big] = 0.
	\end{equation*} 
	The fact that $P_V (Y)$ is Gaussian is an immediate consequence of the equation $\inpro{P_V Y}{v} = \inpro{Y}{v}$. By~\eqref{eq:cylindrical-rv-from-trace-version}, we obtain
	\begin{equation*}
		\inpro{Y}{\phi_k} = \sum_{j=1}^\infty  \inpro[\dot{H}^{-\zeta}]{Y}{A^{\zeta} e_j}\inpro{\phi_k}{e_j}
	\end{equation*} 
	for $k = 1, \ldots, N$. Since $\sigma(P_V (Y)) = \sigma(\inpro{Y}{\phi_1}, \ldots, \inpro{Y}{\phi_N})$, the fact that~$\sigma(P_V (Y)) \subset \sigma(Y)$ follows from~\eqref{eq:sigma-algebra-X}.
\end{proof}
\begin{remark}
	If $Y_1, Y_2$ are two independent cylindrical Gaussian random variables on $H$ (in the sense that $\inpro{Y_1}{u}$ and $\inpro{Y_2}{v}$ are independent for all $u,v \in H$), then $aY_1 + bY_2$ is a cylindrical Gaussian random variable on $H$ for all $a,b \in \R$ and $P_V (aY_1 + bY_2) = a P_V Y_1 + b P_V Y_2$ $\IP$-a.s. Note also, that if $Y$ is an $H$-valued Gaussian random variable, $P_V Y$ (as defined above for the cylindrical random variable induced by $Y$) coincides with the usual orthogonal projection of $Y$ onto $V$.
\end{remark}

\begin{remark}	 
	\label{rem:proj-cyl-obs-well-def}
	Equivalently, we may define $P_V Y$ in Lemma~\ref{lem:fin-dim-proj-well-def} by 
	\begin{equation}
		\label{eq:spectral-expansion-PVY}
		P_V Y := \sum_{j=1}^\infty \inpro[\dot{H}^{-\zeta}]{Y}{f_j} P_V f_j,
	\end{equation}
	where we consider $Y$ as a Gaussian random variable on $\dot{H}^{-\zeta}$ with $(f_j)_{j=1}^\infty \subset \Cov(Y)^{1/2}(\dot{H}^{-\zeta}) \subset H$ being the eigenbasis of $\dot{H}^{-\zeta}$ corresponding to $\Cov(Y) \in \cL_1(\dot{H}^{-\zeta})$. The sum converges to a Gaussian random variable in $L^2(\Omega,H)$ with $\E[\norm{P_V Y}^2] = \norm[\cL_2(\Cov(Y)^{1/2}(H),H)]{P_V}^2 < \infty$, since $P_V$ has finite-dimensional range. That this definition is equivalent to that of Lemma~\ref{lem:fin-dim-proj-well-def} can be seen by an argument similar to that leading up to~\eqref{eq:cylindrical-rv-from-trace-version}. Note that the Gaussian random variables in the sum \eqref{eq:spectral-expansion-PVY} are independent. Therefore, the convergence holds also $\IP$-almost surely \cite[Corollary~6.4.4]{HvNVW17b}. With $\mu_Y$ denoting the image measure of $Y$, we obtain a Borel subspace $\cM_{V,Y} \subset \dot{H}^{-\zeta}$ with $\mu_Y(\cM_{V,Y}) = 1$ such that $P_V$ is well-defined and linear on $\cM_{V,Y}$. When restricted to $H$, this operator coincides with the ordinary orthogonal projection $P_V \colon H \to V$. 
\end{remark}

Next, we let $S_V = (S_V(t))_{t \in [0,T]}$ be a $C_0$-semigroup of symmetric linear operators on $V$. We use this to define an approximation $X^{x,y}_V$ of $X^{x,y}$ by
\begin{equation}
	\label{eq:abstract-discrete-bridge-def}
	X^{x,y}_V(t) := X_V^x(t) - Q_{V}(t) S_V(T-t) (Q_{V}(T) + \tilde Q_V  )^{-1}(X_V^x(T) + P_V Z-P_V y)
\end{equation}
for $t \in [0,T]$ and $y$ in $\dot{H}^{\eta}$. Here 
\begin{equation}
	\label{eq:abstract-discrete-spde-def}
	X_V^x(t) := S_V(t)P_Vx+\int_0^t S_V(t-s)P_V \dd W(s)
\end{equation} is a spatially semidiscrete approximation of the solution $X^x(t)$ to the SPDE~\eqref{eq:SPDE} at time $t \in [0,T]$ started at $x \in H$. Since $V$ is finite-dimensional, the stochastic integral is well-defined and $X_V^x$ has a continuous modification. The operators~$\tilde Q_V$ and~$Q_{V}(t)$ are given by $\tilde Q_V = P_V \tilde Q P_V$ and
\begin{equation*}
	{Q}_{V}(t) := \Cov(X^0_V(t)) = \int^t_0 S_V(s) P_V Q S_V(s) \dd s
\end{equation*}
for $t \in [0,T]$. The $V$-valued random variable $P_V Z$ is given by Lemma~\ref{lem:fin-dim-proj-well-def} and $P_V y$ is well-defined for $\mu_T$-a.e. $y \in \dot{H}^\eta$, c.f.\ Remark~\ref{rem:proj-cyl-obs-well-def}. 

The rest of this section consists of two lemmas that will help us obtain error rates for this SPDE approximation when the space $V$ and semigroup approximation $S_V$ are further specified in the following sections. To obtain higher convergence rates of the approximation $X_V^{x,y}$, we must introduce another regularity parameter $\rho$.
\begin{assumption}
	\label{ass:rho}
	There is a $\rho \ge 0$ such that 
	\begin{equation*}
		\int^T_0 \norm[\cL(Q^{\frac{1}{2}}(H),\dot{H}^\rho)]{S(T-t)}^2 \dd t = \int^T_0 \norm[\cL(H)]{A^{\frac{\rho}{2}}S(T-t)Q^{\frac{1}{2}}}^2 \dd t < \infty
	\end{equation*}	
\end{assumption}
Under~Assumption~\ref{ass:Q}, this is fulfilled for all $\rho \in [0,\beta)$ but, depending on the properties of $Q$ and $A$, it is often true for higher values of $\rho$. Note also, that under this assumption $Q(T)$ extends to $\dot{H}^{-\rho}$ and
\begin{equation}
	\label{eq:QT-rho-bound}
	\norm[\cL(\dot{H}^{-\rho},\dot{H}^{\rho})]{Q(T)} = \norm{A^{\frac{\rho}{2}}Q(T)A^{\frac{\rho}{2}}} < \infty.
\end{equation}

\begin{example}
	%\label{ex:Q-rho-reg}
	We comment on this assumption in the setting of Example~\ref{ex:Q-reg}. In the case of space-time white noise, we directly see from~\eqref{eq:semigroup-bound} that Assumption~\ref{ass:rho} is satisfied for all $\rho < 1$. Next, we consider spatially homogeneous noise, with a kernel $q$ fulfilling the same condition with $\sigma > d/2$ as in Example~\ref{ex:Q-reg}. From the fact that $Q^{1/2}(H) \hookrightarrow H^\sigma$, $\eqref{eq:sobolev_id_1}$ and $\eqref{eq:sobolev_id_2}$ we obtain that Assumption~\ref{ass:rho} is satisfied for all $\rho < 1 + \min(\sigma,1/2)$ in the case of Dirichlet boundary conditions and for all $\rho < 1 + \min(\sigma,3/2)$ in the case of Neumann boundary conditions.
\end{example}

We also need an assumption that relates to the properties of the discrete spaces $(V_i)_{i \in \cI}$. 

\begin{assumption} 
	\label{ass:abstract-approx}
	The following three statements hold true.
	\begin{enumerate}[label=(\roman*)]
		\item \label{ass:abstract-approx:findim-bound-1} There is a constant $C< \infty$ such that
		\begin{equation*}
			\sup_{\substack{V \in (V_i)_{i \in \cI} \\ t \in [0,T]}} \norm[\cL(H)]{S_V(t)P_V} 
			< C.
		\end{equation*}
		\item \label{ass:abstract-approx:findim-bound-2} There is a constant $C< \infty$ such that
		\begin{equation*}
			\sup_{V \in (V_i)_{i \in \cI}} \E[\norm[\cC_T]{X_V^0}^2] 
			< C.
		\end{equation*}
		\item \label{ass:abstract-approx:commutativity}
		For all $V \in (V_i)_{i \in \cI}$, the orthogonal projection $P_V$ commutes with $\tilde Q$.
	\end{enumerate}
\end{assumption}

\begin{remark}
	The estimate in Assumption~\ref{ass:abstract-approx}\ref{ass:abstract-approx:findim-bound-2} holds with $2$ replaced by any $p\ge1$, since $X^0_V$ is Gaussian.
\end{remark}

\begin{remark}
	\label{rem:commutativity}
	Clearly, the commutativity of $P_V$ with $\tilde Q$ implies that $\tilde Q(V) \subset V$. Since $\tilde Q(V) \subset V$ if and only if $\tilde Q(V^\perp) \subset V^\perp$ (as $\tilde Q$ is assumed to be self-adjoint), one can use the identity $\tilde Q = \tilde Q P_V  + \tilde Q P_{V^\perp}$ to see that if $\tilde Q(V) \subset V$, then $P_V$ and $\tilde Q$ commute. Therefore, Assumption~\ref{ass:abstract-approx}\ref{ass:abstract-approx:commutativity} is equivalent to the statement: for all $V \in (V_i)_{i \in \cI}$, $\tilde Q(V) \subset V$. 
\end{remark}

The next two results, Lemmas~\ref{lem:abstract-error-lem-1} and~\ref{lem:abstract-error-lem-2}, are key to our main results in Sections~\ref{sec:spectral-approximation} and~\ref{sec:approximation-fem}. In both lemmas, we derive bounds on approximations of conditional expectations. These are applied to approximations of the term $\hat{X}^0$ in the decomposition~\eqref{eq:prop:bridge-existence:existence-2}.

In the first of these key lemmas, the expression $P_V (X^0(T) + Z)$ should be understood in the sense of Lemma~\ref{lem:fin-dim-proj-well-def} and $X^0(T)+Z$ as a random variable in $\dot{H}^{\eta}$. In the proof we make use of properties of the cross-covariance operator
\begin{equation*}
	{Q}_{V,H}(t) := \Cov(X^0_V(t),X^0(t)) = \int^t_0 S_V(s) P_V Q S(s) \dd s,
\end{equation*}
and we write ${Q}_{H,V}(t):={Q}_{V,H}(t)^*=\Cov(X^0(t),X^0_V(t))$ for $t \in [0,T]$.

\begin{lemma}
	\label{lem:abstract-error-lem-1}
	Let Assumptions~\ref{ass:Q}, \ref{ass:Z}, \ref{ass:rho} and~\ref{ass:abstract-approx} be satisfied. If $\alpha \le \rho$, then for all $p \ge 1$, there exists a constant $C < \infty$ such that for all $V \in (V_i)_{i \in \cI}$,
	\begin{align*}
		&\norm[L^p(\Omega,\cC_T)]{\E[X^{0} | X^0(T) + Z] - \E[X_V^{0} | P_V (X^0(T) + Z)]} \\ 
		&\quad\le C \Bigg( \int_{0}^{T} t^{-\epsilon} \norm[\cL_2(Q^{\frac{1}{2}}(H),H)]{S(t)-S_V(t)P_V}^2 \dd t + \sup_{t \in [0,T]} \norm[\cL(\dot{H}^r,H)]{S(t)-S_V(t)P_V}^2 \\
		&\hspace{6em}+ \norm[\cL(\dot{H}^{-\alpha},\dot{H}^{-\rho})]{I-P_V}^2 \Bigg)^{\frac{1}{2}}.
	\end{align*}
\end{lemma}
\begin{proof}
	For simplicity we assume $\eta<0$, the proof for the case that~$\eta\ge0$ is similar. The goal is to obtain the stated bound on $\norm[L^p(\Omega,\cC_T)]{\cE_V}$, where
	\begin{equation*}
		\cE_V = \E[X^{0} | X^0(T) + Z] - \E[X_V^{0} | P_V (X^0(T) + Z)] = \E[X^{0} - \E[X_V^{0} | P_V (X^0(T) + Z)] | X^0(T) + Z].
	\end{equation*}
	Here, \cite[Section~2.4.1]{VTC87} was used in the second equality, justified by~\eqref{eq:sigma-algebra-X} and the fact that $\sigma(P_V (X^0(T) + Z)) \subset \sigma(X^0(T) + Z)$. The proof is divided into several parts. In Part~\ref{lem:abstract-error-lem-1:pf:1}, we establish that $(X^{0} - \E[X_V^{0} | P_V (X^0(T) + Z)],X^0(T) + Z)$ are jointly Gaussian. By Theorem~\ref{thm:conditional-gaussian}, therefore, $\cE_V$ is an $H_T$-valued Gaussian random variable. We show, in Parts~\ref{lem:abstract-error-lem-1:pf:2} and~\ref{lem:abstract-error-lem-1:pf:3}, that $\norm[L^2(\Omega,\cC_T)]{\cE_V} < \infty$. Then it follows that the law of $\cE_V$ is Gaussian in $\cC_T$ too, since $\cC_T$ is dense in $H_T$. Therefore, we only consider the bound on $\norm[L^p(\Omega,\cC_T)]{\cE_V}$ for $p = 2$, since the $L^p(\Omega,\cC_T)$-norm of a $\cC_T$-valued Gaussian random variable can be bounded by a constant, depending on $p$, times the $L^2(\Omega,\cC_T)$ norm \cite[Proposition~3.14]{H09}.
	\begin{proofpart}
		\label{lem:abstract-error-lem-1:pf:1} % 1
		By considering sequences of elementary integrands and arguing as in the proof of \cite[Theorem~5.2(iii)]{DPZ14}, it can be seen, using also the independence of $Z$ and $W$, that $(X_V^{0},(X^0(T) + Z))$ is a pair of jointly Gaussian random variables. From this it follows that $(X^{0} - \E[X_V^{0} | P_V (X^0(T) + Z)],X^0(T) + Z)$ are jointly Gaussian. To see this, first note that with $$B_V := \Cov(X_V^0,P_V(X^0(T)+Z)) \Cov(P_V(X^0(T)+Z))^{-1} \in \cL(V,H_T),$$ we obtain from Theorem~\ref{thm:conditional-gaussian} that $\E[X_V^{0} | P_V (X^0(T) + Z)] = B_V P_V(X^0(T)+Z)$. Therefore, for $(u, v ) \in H_T \oplus \dot{H}^{\eta}$,
		\begin{align*}
			&\inpro[H_T]{X^{0} - \E[X_V^{0} | P_V (X^0(T) + Z)]}{u} + \inpro[\dot{H}^{\eta}]{X^0(T) + Z}{v} \\
			&\quad= \inpro[H_T]{X^{0}}{u} - \inpro{P_V (X^0(T) + Z)}{B_V^*u} + \inpro[\dot{H}^{\eta}]{X^0(T) + Z}{v} \\
			&\quad= \inpro[H_T]{X^{0}}{u} - \inpro{X^0(T)}{B_V^*u} -\inpro{Z}{B_V^*u} + \inpro[\dot{H}^{\eta}]{X^0(T)}{v} + \inpro[\dot{H}^{\eta}]{Z}{v}\\
			&\quad= \inpro[H_T]{X^{0}}{u} - \inpro{X^0(T)}{B_V^*u - A^{\eta}v} -\sum_{j=1}^{\infty} \inpro[\dot{H}^{\eta}]{Z}{A^{- \eta} e_j} \inpro{e_j}{B_V^*u} + \inpro[\dot{H}^{\eta}]{Z}{v},
		\end{align*}
		so this claim follows from the fact that the pair $(X^0,X^0(T))$ is jointly Gaussian and independent of $Z$. Here we made use of~Lemma~\ref{lem:fin-dim-proj-well-def} and \eqref{eq:cylindrical-rv-from-trace-version}.
	\end{proofpart}
	\begin{proofpart} % 2
		\label{lem:abstract-error-lem-1:pf:2}
		We apply a factorization argument to obtain a first bound on $\norm[L^2(\Omega,\cC_T)]{\cE_V}$. We write 
		\begin{equation*}
			\cK_{V,H} := \left(\Cov(X^0(T)+Z)^{-\frac{1}{2}} \Cov\big(X^0(T)+Z,\E[X_V^{0} | P_V (X^0(T) + Z)]\big)\right)^*
		\end{equation*}
		so that, by Theorem~\ref{thm:conditional-gaussian}, 
		\begin{equation*}
			\cE_V = \E[X^{0} - \E[X_V^{0} | P_V (X^0(T) + Z)] | X^0(T) + Z] = (\cK - \cK_{V,H})\Cov(X^0(T)+Z)^{-\frac{1}{2}} (X^0(T)+Z).
		\end{equation*}
		In light of~\eqref{eq:cylindrical-rv-from-trace-version} and~\eqref{eq:observation-w-noise-cov}, we obtain for $v \in V$ and $u \in \dot{H}^{\eta}$,
		\begin{align*}
			&\inpro[\dot{H}^{\eta}]{\Cov(X^0(T)+Z,P_V (X^0(T) + Z))v}{u}\\ 
			&\quad= \E\left[ \inpro[\dot{H}^{\eta}]{X^0(T)+Z}{u}\inpro{X^0(T)+Z}{v} \right] \\
			&\quad= \sum_{j=1}^\infty \E\left[ \inpro[\dot{H}^{\eta}]{X^0(T)+Z}{u}\inpro[\dot{H}^{\eta}]{X^0(T)+Z}{A^{-\eta}e_j} \right] \inpro{v}{e_j} \\ &\quad=  \sum_{j=1}^\infty \inpro[\dot{H}^{\eta}]{(Q(T) + \tilde Q) A^{\eta}A^{-\eta}e_j}{u} \inpro{v}{e_j}
			\\ &\quad=  \sum_{j=1}^\infty \inpro[\dot{H}^{\eta}]{(Q(T) + \tilde Q)e_j}{u} \inpro{v}{e_j} = \inpro{(Q(T) + \tilde Q) A^{\eta}u}{v} = \inpro[\dot{H}^{\eta}]{(Q(T) + \tilde Q)v}{u}
		\end{align*}
		so that
		\begin{equation*}
			\Cov(X^0(T)+Z,P_V (X^0(T) + Z)) = (Q(T) + \tilde Q)|_V \in \cL(V,\dot{H}^{\eta}),
		\end{equation*}
		which yields 
		\begin{equation*}
			\Cov(P_V (X^0(T) + Z),X^0(T)+Z) = P_V (Q(T) + \tilde Q) A^{\eta} \in \cL(\dot{H}^{\eta},V).
		\end{equation*}
		The interchange of expectation and summation in the second step above is justified by the convergence in $L^2(\Omega,\R)$ of the sum 
		\begin{equation*}
			\sum_{j=1}^\infty  \inpro[\dot{H}^{\eta}]{X^0(T)+Z}{A^{-\eta}e_j} \inpro{v}{e_j}.
		\end{equation*}
		Using the first part of  Proposition~\ref{app:inverse-bound} and the fact that $\dot{H}^{\alpha} \hookrightarrow \tilde{Q}^{1/2}(H)$, we now note that 
		\begin{equation}
			\label{eq:discrete-invertibility}
			\norm{(P_V Q(T)P_V +\tilde Q_V)^{\frac{1}{2}}P_Vv}^2= \norm{Q(T)^{\frac{1}{2}}P_Vv}^2 + \norm{\tilde{Q}^{\frac{1}{2}}P_Vv}^2 \ge C \norm{A^{-\frac{\alpha}{2}}P_V v}^2,
		\end{equation}
		for $v \in H$, where the constant $C$ does not depend on the specific choice of $V$. Therefore, the operators $(P_V Q(T)P_V +\tilde Q_V)^{1/2}$ and $P_V Q(T)P_V +\tilde Q_V$ are invertible on the finite-dimensional space $\cL(V)$. 
		Hence, by an argument similar to the one for $\Cov(P_V (X^0(T) + Z),X^0(T)+Z)$, we see that 
		\begin{equation*}
			B_V = {Q}_{V,H}(\cdot) S(T-\cdot) (P_V Q(T)P_V +\tilde Q_V)^{-1}.
		\end{equation*}
		Here ${Q}_{V,H}(\cdot) S(T-\cdot) \colon H \to H_T$ is defined by $({Q}_{V,H}(\cdot)S(T-\cdot) u) (t) = {Q}_{V,H}(t)S(T-t) u$ for $u \in H$, $t \in [0,T]$. This then implies that
		\begin{equation*}
			\cK_{V,H} = \left(((Q(T) + \tilde Q) A^{\eta})^{-\frac{1}{2}} (Q(T) + \tilde Q) (P_V Q(T)P_V +\tilde Q_V)^{-1} ({Q}_{V,H}(\cdot) S(T-\cdot))^* \right)^*.
		\end{equation*}
		
		Next, let $(f_j)_{j=1}^\infty \subset \dot{H}^{\eta}$ denote the eigenbasis of $\Cov(X^0(T) + Z)$ and $(z_j)_{j=1}^\infty$ a sequence of i.i.d.\ Gaussian random variables. Then, by Lemma~\ref{lem:conditional-lemma} and~\eqref{eq:gamma-rad-norm}, 
		\begin{equation*}
			\norm[L^2(\Omega,\cC_T)]{\cE_V} = \E\Bigg[\Bignorm[\cC_T]{\sum_{j=1}^\infty z_j (\cK - \cK_{V,H}) f_j}^2\Bigg] = \norm[\gamma(\dot{H}^{\eta},\cC_T)]{\cK - \cK_{V,H}}^2.
		\end{equation*}
		
		We now recall from the proof of Lemma~\ref{lem:K-is-radonifying} that $\cK=\cJ^1 \cJ^0$. Similarly, $\cK_{V,H}$ can be factored into $\cK_{V,H} = \cJ_{V}^1 \cJ_{V,H}^0$, where $\cJ_{V}^1 \colon H_T \to \cC_T$ is given by 
		\begin{equation}
			\label{eq:JV1}
			(\cJ_{V}^1 v)(t) := \int^t_0 S_V(t-s) P_V Q^{\frac{1}{2}} v(s) \dd s
		\end{equation}
		while $\cJ_{V,H}^0 \colon \dot{H}^{\eta} \to H_T$ is given by
		\begin{equation}
			\label{eq:JVH-def}
			\begin{split}
				(\cJ_{V,H}^0 v) (t) :=& Q^{\frac{1}{2}} S(T-t) (P_V Q(T)P_V +\tilde Q_V)^{-1} P_V (Q(T) + \tilde Q)A^{\eta} ((Q(T) + \tilde Q) A^{\eta})^{-1/2} v \\
				=& Q^{\frac{1}{2}} S(T-t) (P_V Q(T)P_V +\tilde Q_V)^{-1} P_V ((Q(T) + \tilde Q) A^{\eta})^{1/2} v
			\end{split}
		\end{equation}
		for $v \in \dot{H}^{\eta}$ and almost every $t \in [0,T]$. By invariance of the reproducing kernel Hilbert space, $((Q(T) + \tilde Q) A^{\eta})^{1/2}(\dot{H}^{\eta}) = ((Q(T) + \tilde Q))^{1/2}(H) \subset H$, so that $\cJ_{V,H}^0$ is well-defined as an operator in $\cL(\dot{H}^{\eta},H_T)$.
		We use these factorizations and the ideal property of the $\gamma$-radonifying norm 
		to make the split
		\begin{equation}
			\label{eq:lem:abstract-error-lem-1:split}
			\norm[\gamma(\dot{H}^{\eta},\cC_T)]{\cK - \cK_{V,H}} \le \norm[\gamma(H_T,\cC_T)]{\cJ^1-\cJ_{V}^1} \norm[\cL(\dot{H}^{\eta},H_T)]{\cJ^0} + \norm[\gamma(H_T,\cC_T)]{\cJ_V^1} \norm[\cL(\dot{H}^{\eta},H_T)]{\cJ^0-\cJ_{V,H}^0}. 
		\end{equation} 
	\end{proofpart}	
	\begin{proofpart} % 2
		\label{lem:abstract-error-lem-1:pf:3}	
		We now derive bounds for the four terms of~\eqref{eq:lem:abstract-error-lem-1:split}. Note that the fact that $\norm[H_T]{\cJ^0 v} < \infty$ was shown in Lemma~\ref{lem:K-is-radonifying}, we move on to show that $\sup_{V \in (V_i)_{i \in \cI}} \norm[\gamma(H_T,\cC_T)]{\cJ_V^1} < \infty$. As in the proof of Lemma~\ref{lem:K-is-radonifying}, we have $\cJ_V^1(H_T) = {\bar Q_V}^{1/2}(H_T)$, where $\bar Q_V$ is the covariance operator of $X_V^0 \in L^2(\Omega,H_T)$ and $\bar Q_V^{1/2}(H_T)$ is the reproducing kernel Hilbert space of $X^0_V$ on $\cC_T$. Thus, by~\eqref{eq:gamma-rad-expectation-property} and Assumption~\ref{ass:abstract-approx}\ref{ass:abstract-approx:findim-bound-2}, there exists a constant $C<\infty$, independent of the specific choice of $V$, such that
		\begin{equation}
			\label{eq:JV1-bound}
			\norm[\gamma(H_T,\cC_T)]{\cJ_V^1}^2 = \E\left[\norm[\cC_T]{X^0_V(t)}^2\right] < C.
		\end{equation}
		
		For the term $\norm[\gamma(H_T,\cC_T)]{\cJ^1-\cJ_{V}^1}$, we note that by a straightforward calculation along with~\cite[Corollary~B.5]{DPZ14}, $(\cJ^1-\cJ_{V}^1)(H_T) = \Cov(X^0-X^0_V)^{1/2}(H)$. By the same arguments as for the term~$\norm[\gamma(H_T,\cC_T)]{\cJ_V^1}$, we therefore obtain 
		\begin{align*}
			\norm[\gamma(H_T,\cC_T)]{\cJ^1-\cJ_{V}^1}^2 = \E[\norm[\cC_T]{X^0-X_V^0}^2] 			&\lesssim \int_{0}^{T} t^{-\epsilon} \norm[\cL_2(Q^{\frac{1}{2}}(H),H)]{S(t)-S_V(t)P_V}^2 \dd t \\
			&\hspace{3em}+ \sup_{t \in [0,T]} \norm[\cL(\dot{H}^r,H)]{S(t)-S_V(t)P_V}^2,
		\end{align*}
		where we made use of Proposition~\ref{app:prop:cont-error} in the last step.
		
		From~\eqref{eq:sum-covariance-bounded-below} and Proposition~\ref{app:inverse-bound}, we find that
		\begin{equation}
			\label{eq:lem:abstract-error-lem-1:inverse-bound-1}
			\norm[\cL(\dot{H}^\alpha,\dot{H}^{\eta})]{\big((Q(T)  + \tilde Q)A^{\eta}\big)^{-\frac{1}{2}}} \le \norm[\cL(\dot{H}^\alpha,\dot{H}^{\eta})]{A^{-\frac{\eta - \alpha}{2}}} = 1.
		\end{equation}	
		Using this result along with the symmetry of $((Q(T) + \tilde Q) A^{\eta})^{-\frac{1}{2}}$ on $\dot{H}^{\eta}$, we obtain that for $u \in \dot{H}^{\alpha}$,
		\begin{align}
			\notag\norm[\dot{H}^{2 \eta-\alpha}]{((Q(T) + \tilde Q) A^{\eta})^{-\frac{1}{2}}u} &= \sup_{ \substack{ v \in \dot{H}^{2 \eta-\alpha} \\ \norm[\dot{H}^{2 \eta - \alpha}]{v} = 1}} \Big|\inpro[\dot{H}^{2 \eta-\alpha}]{((Q(T) + \tilde Q) A^{\eta})^{-\frac{1}{2}} u}{v}\Big| \\
			\label{eq:lem:abstract-error-lem-1:inverse-bound-2} &= \sup_{ \substack{ v \in \dot{H}^{2 \eta-\alpha} \\ \norm[\dot{H}^{2 \eta - \alpha}]{v} = 1}} \Big|\inpro[\dot{H}^{\eta}]{u}{((Q(T) + \tilde Q) A^{\eta})^{-\frac{1}{2}} A^{-\alpha + {\eta}} v}\Big| \\ 
			\notag &\le \norm[\dot{H}^{\eta}]{u} \norm[\cL(\dot{H}^{\alpha},\dot{H}^{\eta})]{((Q(T) + \tilde Q) A^{\eta})^{-\frac{1}{2}}}. 
		\end{align}
		Therefore, $((Q(T) + \tilde Q) A^{\eta})^{-1/2}$ extends to an operator in $\cL(\dot{H}^{\eta},\dot{H}^{2 \eta-\alpha})$, and we retain its notation.
		Finally, by~\eqref{eq:lem:abstract-error-lem-1:inverse-bound-2} and the commutativity $P_V$ and $\tilde Q$ from Assumption~\ref{ass:abstract-approx}\ref{ass:abstract-approx:commutativity}, we find that for $v \in \dot{H}^\eta$, 
		\begin{align*}
			\norm[H_T]{(\cJ^0-\cJ_{V,H}^0)v}^2 &= \int^T_0 \Big\|Q^{\frac{1}{2}} S(T-t) \Big((P_V Q(T)P_V +\tilde Q_V)^{-1} P_V (Q(T) + \tilde Q) - I \Big) \\ 
			&\hspace{4em} \times A^{\eta} ((Q(T) + \tilde Q) A^{\eta})^{-1/2} v\Big\|_{\cL(\dot{H}^{\eta},H)}^2 \dd t \\
			&\lesssim \left(\mathrm{I} + \mathrm{II}\right) \norm{A^{\eta - \frac{\alpha}{2}} ((Q(T) + \tilde Q) A^{\eta})^{-\frac{1}{2}} v}^2 \lesssim  \left(\mathrm{I} + \mathrm{II}\right) \norm[\dot{H}^\eta]{v}^2. 
		\end{align*}
		Here,
		\begin{align*}
			\mathrm{I} := \int^T_0 \Big\|Q^{\frac{1}{2}} S(T-t) (P_V Q(T)P_V +\tilde Q_V)^{-1} P_V Q(T) (I - P_V) A^{\frac{\alpha}{2}}\Big\|_{\cL(H)}^2 \dd t
		\end{align*}
		and
		\begin{align*}
			\mathrm{II} := \int^T_0 \Big\|Q^{\frac{1}{2}} S(T-t) \Big(I-P_V \Big) A^{\frac{\alpha}{2}}\Big\|_{\cL(H)}^2 \dd t.
		\end{align*}
		
		For the first of these two terms, we note that $(P_V Q(T)P_V +\tilde Q_V)^{-1/2}$, the inverse of $(P_V Q(T)P_V +\tilde Q_V)^{1/2} \in \cL(V)$, coincides with the pseudoinverse of $(P_V Q(T)P_V +\tilde Q_V)^{1/2} P_V \in \cL(H)$ when this is restricted to $V$. Since $(A^{-\frac{\alpha}{2}}P_V)^* = P_V A^{-\frac{\alpha}{2}}$, it follows from~\eqref{eq:discrete-invertibility} and Proposition~\ref{app:inverse-bound} that $\norm[\cL(H)]{(P_V Q(T)P_V +\tilde Q_V)^{-\frac{1}{2}} P_V A^{-\frac{\alpha}{2}}}$ is bounded by
		\begin{align*}
			C \norm[\cL(H)]{(P_V A^{-\frac{\alpha}{2}})^{-1} P_V A^{-\frac{\alpha}{2}}} = C \norm[\cL(H)]{P_{\ker(P_V A^{-\frac{\alpha}{2}})^\perp}} \le C,
		\end{align*}
		where $P_{\ker(P_V A^{-\frac{\alpha}{2}})^\perp} \colon H \to \ker(P_V A^{-\frac{\alpha}{2}})^\perp$ is the orthogonal projection onto $\ker(P_V A^{-\frac{\alpha}{2}})^\perp$. As a consequence, by symmetry of the involved operators, we also have 
		\begin{equation*}
			\norm[\cL(H)]{ A^{-\frac{\alpha}{2}}(P_V Q(T)P_V +\tilde Q_V)^{-\frac{1}{2}}P_V} = \norm[\cL(H)]{(P_V Q(T)P_V +\tilde Q_V)^{-\frac{1}{2}}P_V A^{-\frac{\alpha}{2}}} \le C.
		\end{equation*}
		Putting these two inequalities together, it follows that
		\begin{equation*}
			\norm[\cL(\dot{H}^\alpha,\dot{H}^{-\alpha})]{(P_V Q(T)P_V +\tilde Q_V)^{-1}P_V} = \norm[\cL(H)]{A^{-\frac{\alpha}{2}}(P_V Q(T)P_V +\tilde Q_V)^{-1}P_VA^{-\frac{\alpha}{2}}} \le C^2.
		\end{equation*}
		We use this result along with~\eqref{eq:QT-rho-bound} and the fact that $\alpha \le \rho$ to see that
		\begin{equation*}
			\mathrm{I} \lesssim \norm[\cL(H)]{A^{\frac{\alpha}{2}}Q(T) (I - P_V) A^{\frac{\alpha}{2}}}^2 \lesssim \norm[\cL(H)]{A^{-\frac{\rho}{2}}(I-P_V)A^{\frac{\alpha}{2}}}^2, 
		\end{equation*}
		and by Assumption~\ref{ass:rho} we obtain the same bound for $\mathrm{II}$. \qedhere
	\end{proofpart}
\end{proof}

\begin{remark}
	An alternative way of treating the term $\norm[\gamma(\dot{H}^{\eta},\cC_T)]{\cK - \cK_{V,H}}$ above would be to make the split 
	\begin{equation*}
		\norm[\gamma(\dot{H}^{\eta},\cC_T)]{\cK - \cK_{V,H}} \le \norm[\cL(H_T,\cC_T)]{\cJ^1-\cJ_{V}^1} \norm[\cL_2(\dot{H}^{\eta},H_T)]{\cJ^0} + \norm[\gamma(H_T,\cC_T)]{\cJ_V^1} \norm[\cL(\dot{H}^{\eta},H_T)]{\cJ^0-\cJ_{V,H}^0}. 
	\end{equation*} 
	It can be seen that $\norm[\cL_2(\dot{H}^{\eta},H_T)]{\cJ^0} < \infty$ provided that $\alpha \le \beta$. At the cost of this additional assumption, the bound on the error in Lemma~\eqref{lem:abstract-error-lem-1} can then be improved, since $\norm[\cL(H_T,\cC_T)]{\cJ^1-\cJ_{V}^1} \le \norm[\gamma(H_T,\cC_T)]{\cJ^1-\cJ_{V}^1}$.
\end{remark}

For the next lemma, we need an additional assumption. This lemma is only applied to the finite element approximation of Section~\ref{sec:approximation-fem}. For this approximation, the only natural example of an operator $\tilde Q$ that fulfills Assumption~\ref{ass:abstract-approx}\ref{ass:abstract-approx:commutativity} is a multiple of the identity operator $I$. This implies that $\alpha = 0$. With $\alpha = 0$, the assumption is already fulfilled and the proof of the lemma below can be simplified, but we choose to prove the result also for $\alpha >0$, to allow for possible generalizations to other approximation schemes.

\begin{assumption}
	\label{ass:abstract-approx-2}	
	The following two statements hold true. 
	\begin{enumerate}[label=(\roman*)]
		\item \label{ass:abstract-approx-2:Q-alpha} For all $V \in (V_i)_{i \in \cI}$, $V \subset \dot{H}^{\alpha}$ and
		\begin{equation*}
			\sup_{V \in (V_i)_{i \in \cI}} \int^T_0 \norm[\cL(Q^{\frac{1}{2}}(H),\dot{H}^\alpha)]{S_V(T-t)P_V}^2 \dd t = \sup_{V \in (V_i)_{i \in \cI}} \int^T_0 \norm[\cL(H)]{A^{\frac{\alpha}{2}}S_V(T-t)P_VQ^{\frac{1}{2}}}^2 \dd t < \infty.
		\end{equation*}
		\item \label{ass:abstract-approx-2:projection} The orthogonal projection $P_V$ satisfies 
		\begin{equation*}
			\sup_{V \in (V_i)_{i \in \cI}} \norm[\cL(\dot{H}^\alpha)]{P_V} = \sup_{V \in (V_i)_{i \in \cI}} \norm[\cL(H)]{A^{\frac{\alpha}{2}} P_V A^{-\frac{\alpha}{2}} } < \infty.
		\end{equation*}
	\end{enumerate}
\end{assumption}

\begin{remark}
	%\label{rem:abstract-approx-2:projection-neg}
	Under Assumption~\ref{ass:abstract-approx-2}\ref{ass:abstract-approx-2:projection}, it can be seen that $P_V$ extends to $\dot{H}^{-\alpha}$. We make no notational distinction between $P_V$ and its extension, which coincides with the generalized orthogonal projection, defined by the equation~$\inpro{P_V u}{v} = \inpro{A^{-\alpha/2}u}{A^{\alpha/2}v}$ for all $u \in \dot{H}^{-\alpha}, v \in V$, and fulfills $\norm[\cL(\dot{H}^{-\alpha})]{P_V} = \norm[\cL(\dot{H}^{\alpha})]{P_V}$.
\end{remark}

\begin{lemma}
	\label{lem:abstract-error-lem-2}
	Let Assumptions~\ref{ass:Q}, \ref{ass:Z}, \ref{ass:rho}, \ref{ass:abstract-approx} and~\ref{ass:abstract-approx-2} be satisfied. If $\alpha \le \rho$, then for all $p \ge 1$, there exists a constant $C < \infty$ such that for all $V \in (V_i)_{i \in \cI}$,
	\begin{align*}
		&\norm[L^p(\Omega,\cC_T)]{\E[X_V^{0} | P_V (X^0(T) + Z)] - \E[X_V^{0} | X_V^{0} + P_V Z]} \\ &\quad\le C\left(\int^T_0 \norm[\cL(Q^{\frac{1}{2}}(H),\dot{H}^{\alpha})]{S(T-t) - S_V(T-t)P_V}^2 \dd t\right)^\frac{1}{2}.
	\end{align*}
\end{lemma}
\begin{proof}
	As in the proof of Lemma~\ref{lem:abstract-error-lem-1}, we start by showing the first claim for $p = 2$, i.e., we derive a bound on
	\begin{align*}
		\cE_V :=& \E[X_V^{0} | P_V (X^0(T) + Z)] - \E[X_V^{0} | X_V^0(T) + P_V Z]
	\end{align*}
	in $L^2(\Omega,\cC_T)$. The general claim then follows from the Gaussian law of this random variable, which follows from a similar argument as in the proof of Lemma~\ref{lem:abstract-error-lem-1}. The proof is split into two parts. In the first part, we use a factorization argument to derive a bound on $\norm[L^2(\Omega,\cC_T)]{\cE_V}$ which we further analyze in the second part.
	\begin{proofpart}
		We recall from the proof of Lemma~\ref{lem:abstract-error-lem-1} that
		\begin{equation*}
			\E[X_V^{0} | P_V (X^0(T) + Z)] = Q_{V,H}(\cdot)S(T-\cdot) (P_V Q(T) P_V + \tilde Q_V)^{-1} P_V (X^0(T) + Z).
		\end{equation*}
		and note that, as another consequence of Theorem~\ref{thm:conditional-gaussian},
		\begin{align*}
			\E[X_V^{0} | X_V^0(T) + P_V Z] &= \Cov(X_V^0,X_V^0(T) + P_V Z) \Cov(X_V^0(T) + P_V Z)^{-1} \\
			&=  {Q}_{V}(\cdot)S_V(T-\cdot) ({Q}_V(T) + \tilde Q_V)^{-1} (X_V^0(T) + P_V Z).
		\end{align*}
		By the same factorization argument as in the proofs of Lemmas~\ref{lem:K-is-radonifying} and~\ref{lem:abstract-error-lem-1}, we may write 
		\begin{equation*}
			\E[X_V^{0} | P_V (X^0(T) + Z)]  = \cJ^1_V \cJ^0_V P_V (X^0(T) + Z).
		\end{equation*}
		and 		
		\begin{equation*}
			\E[X_V^{0} | X_V^0(T) + P_V Z] = \cJ^1_V \tilde \cJ^0_V (X_V^0(T) + P_V Z).
		\end{equation*}
		Here $\cJ^1_V$ is given by~\eqref{eq:JV1} while $\cJ^0_V, \tilde \cJ^0_V \in \cL(H,H_T)$ are given by  
		\begin{equation*}
			(\cJ^0_V v)(t) := Q^{\frac{1}{2}} S(T-t) (P_V Q(T) P_V + \tilde Q_V)^{-1} P_V v
		\end{equation*}
		and
		\begin{equation*}
			(\tilde \cJ^0_V v)(t) := Q^{\frac{1}{2}} S_V(T-t) ({Q}_V(T) + \tilde Q_V)^{-1} P_V v
		\end{equation*}
		for $v \in H$ and almost every $t \in [0,T]$. The fact that $V$ is finite-dimensional ensures that these operators are well-defined. Let us write
		\begin{align*}
			\tilde \cE_V :=& \tilde \cJ^0_V (X_V^0(T) + P_V Z) - \cJ^0_V P_V (X^0(T) + Z) \\
			=& \tilde \cJ^0_V P_V (X_V^0(T) - X^0(T)) + (\tilde \cJ^0_V - \cJ^0_V) P_V (X^0(T) + Z).
		\end{align*}
		Since $\Cov(\cE_V) = (\cJ^1_V\Cov(\tilde \cE_V)^{1/2})(\cJ^1_V\Cov(\tilde \cE_V)^{1/2})^*$, Corollary~B.5 in \cite{DPZ14} yields the result $\cJ^1_V\Cov(\tilde \cE_V)^{1/2}(H_T) = \Cov(\cE_V)^{1/2}(H_T)$. Note also, that $(\cJ^1_V\Cov(\tilde \cE_V)^{1/2} f_j)_{j=1}^\infty$ is an orthonormal basis of $\cJ^1_V\Cov(\tilde \cE_V)^{1/2}(H_T)$ when $(f_j)_{j=1}^\infty$ is an orthonormal basis of $H_T$. Using this, along with~\eqref{eq:gamma-rad-expectation-property}, the ideal property of the $\gamma$-radonifying operators and the bound~\eqref{eq:JV1-bound} obtained in the proof of Lemma~\ref{lem:abstract-error-lem-1}, we see that, for an arbitrary sequence $(z_j)_{j=1}^\infty$ of i.i.d.\ Gaussian random variables,
		\begin{align*}
			\norm[L^2(\Omega,\cC_T)]{\cE_V}^2 &= \norm[\gamma(\Cov(\cE_V)^{\frac{1}{2}}(H_T), \cC_T )]{I_{\Cov(\cE_V)^{\frac{1}{2}}(H_T) \hookrightarrow \cC_T}}^2 \\
			&= \norm[\gamma(\cJ^1_V\Cov(\tilde \cE_V)^{\frac{1}{2}}(H_T), \cC_T )]{I_{\cJ^1_V\Cov(\tilde \cE_V)^{\frac{1}{2}}(H_T) \hookrightarrow \cC_T}}^2 \\
			&= \E\Bigg[\Bignorm[\cC_T]{\sum_{j=1}^\infty z_j \cJ^1_V\Cov(\tilde \cE_V)^{\frac{1}{2}} f_j}^2\Bigg] = \norm[\gamma(H_T,\cC_T)]{\cJ^1_V\Cov(\tilde \cE_V)^{\frac{1}{2}}}^2 \lesssim \norm[\cL(H_T)]{\Cov(\tilde \cE_V)}.
		\end{align*}
	\end{proofpart}
	
	\begin{proofpart}
		We split the term $\norm[\cL(H_T)]{\Cov(\tilde \cE_V)}$ using the independence of $Z$ and $W$, by  
		\begin{align*}
			\norm[\cL(H_T)]{\Cov(\tilde \cE_V)} &\le \norm[\cL(H_T)]{\tilde \cJ^0_V \Cov(X_V^0(T) - X^0(T)) (\tilde \cJ^0_V)^* } \\
			&\quad + 2 \norm[\cL(H_T)]{\tilde \cJ^0_V \Cov(X_V^0(T) - X^0(T),X^0(T)) (\tilde \cJ^0_V - \cJ^0_V)^* } \\
			&\quad + \norm[\cL(H_T)]{(\tilde \cJ^0_V - \cJ^0_V) \Cov(P_V (X^0(T) + Z)) (\tilde \cJ^0_V - \cJ^0_V)^*} =: \mathrm{I} + \mathrm{II} + \mathrm{III}.
		\end{align*}
		Next, for $v \in H$, we may write 
		\begin{align*}
			\norm[H_T]{\tilde \cJ^0_V v - \cJ^0_V v}^2 &= \int^T_0 \big\|Q^{\frac{1}{2}} \Big(S_V(T-t) (Q_{V}(T) + \tilde Q_V)^{-1} (Q_V(T) - P_V Q(T) P_V) \\ 
			&\hspace{4em}- \big(S(T-t) - S_V(T-t)\big) \Big) (P_V Q(T) P_V + \tilde Q_V)^{-1} P_V v\big\|^2 \dd t. 
		\end{align*}
		We already noted in the proof of Lemma~\ref{lem:abstract-error-lem-1} that as a consequence of Proposition~\ref{app:inverse-bound},
		\begin{equation*}
			%\label{eq:abstract-error-lem-2:bound-3}
			\norm[\cL(\dot{H}^{\alpha},\dot{H}^{-\alpha})]{(P_V Q(T) P_V + \tilde Q_V)^{-1}P_V} \le C,
		\end{equation*}
		where the constant $C$ does not depend on the choice of $V$. In a similar way, we obtain the bound
		\begin{equation*}
			%\label{eq:abstract-error-lem-2:bound-1}
			\norm[\cL(\dot{H}^{\alpha},\dot{H}^{-\alpha})]{(Q_{V}(T) + \tilde Q_V)^{-1}P_V} \le C.
		\end{equation*}
		Using these two results along with Assumptions~\ref{ass:rho} and~\ref{ass:abstract-approx-2}, it follows that 
		\begin{equation*}
			\norm[\cL(\dot{H}^\alpha,H_T)]{\tilde \cJ^0_V - \cJ^0_V}^2 \lesssim
			\norm[\cL(\dot{H}^{-\alpha},\dot{H}^{\alpha})]{Q(T)-Q_V(T)}^2 + \int^T_0 \norm[\cL(Q^{\frac{1}{2}}(H),\dot{H}^\alpha)]{S(T-t) - S_V(T-t)P_V}^2 \dd t. 
		\end{equation*}
		Similarly, there is a constant $C< \infty$, independent of $V$, such that $\norm[\cL(\dot{H}^\alpha,H_T)]{\tilde \cJ^0_V} < C$.
		
		With these estimates in place, we move on to estimate the terms $\mathrm{I}$ and $\mathrm{II}$. Since
		\begin{align*}
			\Cov(X_V^0(T) - X^0(T)) = \int^T_0 (S_V(t)P_V - S(t)) Q (S_V(t)P_V - S(t)) \dd t,
		\end{align*}
		the previously obtained bounds imply that
		\begin{equation*}
			\mathrm{I} \lesssim \int^T_0 \norm[\cL(H)]{Q^{\frac{1}{2}}(S_V(t)P_V - S(t))A^{\frac{\alpha}{2}}}^2 \dd t = \int^T_0 \norm[\cL(H)]{A^{\frac{\alpha}{2}}(S(T-t) - S_V(T-t)P_V)Q^{\frac{1}{2}}}^2 \dd t.
		\end{equation*}
		Similarly, from the fact that 
		\begin{align*}
			\Cov(X_V^0(T) - X^0(T), X^0(T)) = Q_{V,H}(T) -  Q(T),
		\end{align*}
		it follows that
		\begin{align*}
			\mathrm{II} &\lesssim  \norm[\cL(\dot{H}^{-\alpha},\dot{H}^\alpha)]{Q(T)-Q_{V,H}(T)} \\
			&\qquad\times \left(\norm[\cL(\dot{H}^{-\alpha},\dot{H}^{\alpha})]{Q(T)-Q_V(T)}^2 + \int^T_0 \norm[\cL(H)]{A^{\frac{\alpha}{2}}(S(T-t) - S_V(T-t)P_V)Q^{\frac{1}{2}}}^2 \dd t\right)^\frac{1}{2}.
		\end{align*}
		
		Finally, since $\Cov(P_V (X^0(T) + Z)) = (P_V Q(T) P_V + \tilde Q_V)$, we note that for $v \in H$, 
		\begin{align*}
			&\norm[H_T]{(\tilde \cJ^0_V - \cJ^0_V) \Cov(P_V (X^0(T) + Z)) P_V v}^2 \\ &\quad= \int^T_0 \big\|Q^{\frac{1}{2}} \Big(S_V(T-t) (Q_{V}(T) + \tilde Q_V)^{-1} (Q_V(T) - P_V Q(T) P_V) \\ &\hspace{15em}- \big(S(T-t) - S_V(T-t)\big) \Big) P_V v\big\|^2 \dd t. 
		\end{align*}
		Using the estimates obtained above, we therefore have 
		\begin{equation*}
			\mathrm{III} \lesssim \norm[\cL(\dot{H}^{-\alpha},\dot{H}^{\alpha})]{Q(T)-Q_V(T)}^2 + \int^T_0 \norm[\cL(H)]{A^{\frac{\alpha}{2}}(S(T-t) - S_V(T-t)P_V)Q^{\frac{1}{2}}}^2 \dd t.
		\end{equation*}
		
		The proof is now completed by noting that, by H\"older's inequality and Assumption~\ref{ass:rho}, 
		\begin{align*}
			&\norm[\cL(\dot{H}^{-\alpha},\dot{H}^\alpha)]{Q(T)-Q_{V,H}(T)} \\
			&\quad\le \int^T_0 \norm[\cL(H)]{A^{\frac{\alpha}{2}}(S(T-t) - S_V(T-t)P_V)Q^{\frac{1}{2}}} \norm[\cL(H)]{A^{\frac{\alpha}{2}}S(T-t)Q^{\frac{1}{2}}} \dd t \\
			&\quad\lesssim \left(\int^T_0 \norm[\cL(H)]{A^{\frac{\alpha}{2}}(S(T-t) - S_V(T-t)P_V)Q^{\frac{1}{2}}}^2 \dd t\right)^\frac{1}{2},
		\end{align*}
		along with an analogous bound for the term $\norm[\cL(\dot{H}^{-\alpha},\dot{H}^{\alpha})]{Q(T)-Q_V(T)}$, where also Assumption~\ref{ass:abstract-approx-2} is needed.\qedhere
	\end{proofpart}
\end{proof}
	
\section{Spectral approximation under commutative observation noise}
\label{sec:spectral-approximation}

Next, we apply the results of the previous section to a spectral approximation of the SPDE bridge $X^{x,y}$. This is to say, for $N \in \N$, we set $V = V_N :=\mathrm{span}\{e_1, \ldots, e_N\}$, $S_{V_N} := S_N := P_N S = S P_N$, where we write $P_N$ for the projection $P_{V_N}$, $Q_N := Q_{V_N}$, $\tilde Q_N:=\tilde Q_{V_N}$ and we let the approximation $X_N^{x,y}:=X_{V_N}^{x,y}$ be given by~\eqref{eq:abstract-discrete-bridge-def}. Writing also $X^{x}_N$ for $X^{x}_{V_N}$, the spectral approximation~\eqref{eq:abstract-discrete-spde-def} of $X^{x}$, this has in this setting the special property that $X^{x}_N = P_N X^{x}$ for all $x \in H$. Similarly, $Q_N(t) = P_N Q(t) P_N$ for all $t \in [0,T]$. 

\begin{remark}
	%\label{rem:spectral-convergence}
	In this setting, Assumption~\ref{ass:abstract-approx}\ref{ass:abstract-approx:findim-bound-1} is automatically satisfied, Assumption~\ref{ass:abstract-approx}\ref{ass:abstract-approx:findim-bound-2} is implied by Assumption~\ref{ass:Q} while Assumption~\ref{ass:abstract-approx}\ref{ass:abstract-approx:commutativity} is equivalent to $\tilde Q$ and $A$ sharing a common eigenbasis.
\end{remark}

\begin{remark}
	\label{rem:tilde-mu-spectral}
	Under Assumption~\ref{ass:abstract-approx}\ref{ass:abstract-approx:commutativity}, write $(\tilde \mu_j)_{j=1}^\infty$ for the sequence of eigenvalues of $\tilde{Q}$ corresponding to the eigenbasis $(e_j)_{j=1}^\infty$. Then, by Proposition~\ref{app:inverse-bound},  Assumption~\ref{ass:Z}\ref{ass:Z:alpha} is equivalent to the existence of a constant $C > 0$ such that $\inf_j \lambda_j^{\alpha} \tilde \mu_j \ge C$. 
\end{remark}

We now derive error bounds for $X^{x,y}_N$ as $N \to \infty$. They are obtained from the fact that
\begin{equation}
	\label{eq:projection-error}
	\norm[\cL(H)]{(I-P_N)A^{-r}} = \norm[\cL(H)]{A^{-r}(I-P_N)} \le \lambda_{N+1}^{-r}
\end{equation}
for all $r \ge 0$ and $N \in \N$. 

\begin{theorem}
	\label{thm:spectral-convergence}
	Let Assumptions~\ref{ass:Q}, \ref{ass:Z}, \ref{ass:rho} and~\ref{ass:abstract-approx} be satisfied with $\alpha < \rho$ and let $\chi > 0$. Then, for a Borel subset $\cM \subset \dot{H}^{\eta}$ with $\mu_T(\cM) = 1$ and all $x \in \dot{H}^\chi$, $y \in \cM$, $p \ge 1$  and $r < \min(\rho-\alpha,\beta,\chi)$, there is an $M \in \N$ and a constant $C < \infty$ such that for all $N \ge M$,
	\begin{equation*}
		\norm[L^p(\Omega,\cC_T)]{X_N^{x,y} - X^{x,y}} \le C \lambda_{N+1}^{-\frac{r}{2}}.
	\end{equation*}
\end{theorem}

\begin{proof}
	As in Lemma~\ref{lem:abstract-error-lem-1}, the law of $X_N^{x,y} - X^{x,y}$ can be seen to be Gaussian in $H_T$ and therefore also in $\cC_T$. Using~\eqref{eq:prop:bridge-existence:existence-2} and the fact that $X^{0}_N(T) = P_N X^{0}(T)$, we split the error into four parts,
	\begin{equation*}
		\norm[L^p(\Omega,\cC_T)]{X_N^{x,y} - X^{x,y}} =: \mathrm{I} + \mathrm{II} + \mathrm{III} + \mathrm{IV} + \mathrm{V}.
	\end{equation*}
	The first term is given by $\mathrm{I} = \norm[\cC_T]{S(\cdot)x-S_N(\cdot) x}$, the second by $\mathrm{II} = \lrnorm[L^p(\Omega,\cC_T)]{X^0 - P_N X^0}$. With $\cK$ given by~\eqref{eq:kappa} and $\cK_N: V_N \to C([0,T],V_N)$ by
	\begin{equation*}
		(\cK_N v)(t) := Q_N(t) S_N(T-t) (Q_N(T) + \tilde Q_N)^{-\frac{1}{2}} v,
	\end{equation*}
	for $v \in V_N, t \in [0,T]$,
	the third therm is given by 
	\begin{equation*}
		\mathrm{III} = \lrnorm[\cC_T]{\cK ((Q(T)+\tilde Q)A^\eta)^{-\frac{1}{2}} y - \cK_N (Q_N(T)+\tilde Q_N)^{-\frac{1}{2}} P_N y},
	\end{equation*}
	the fourth by
	\begin{equation*}
		\mathrm{IV} = \lrnorm[\cC_T]{\cK ((Q(T)+\tilde Q)A^\eta)^{-\frac{1}{2}} S(T)x - \cK_N (Q_N(T)+\tilde Q_N)^{-\frac{1}{2}} P_N S(T)x },
	\end{equation*}
	and the fifth by
	\begin{align*}
		\mathrm{V} &= \lrnorm[L^p(\Omega,\cC_T)]{\E[X^{0} | X^0(T) + Z] - \E[P_N X^{0} | X_N^0(T) + P_N Z]} \\ 
		&= \lrnorm[L^p(\Omega,\cC_T)]{\E[X^{0} | X^0(T) + Z] - \E[P_N X^{0} | P_N (X^0(T) + Z)]}\\ 
		&= \lrnorm[L^p(\Omega,\cC_T)]{\cK ((Q(T)+\tilde Q)A^\eta)^{-\frac{1}{2}} (X^0(T) + Z) - \cK_N (Q_N(T)+\tilde Q_N)^{-\frac{1}{2}} P_N(X^0(T) + Z)}. 
	\end{align*}
	
	By~\eqref{eq:projection-error}, the fact that $x \in \dot{H}^\chi$ and the strong continuity of $S$, the first term is bounded by 
	\begin{equation*}
		\mathrm{I} = \sup_{t \in [0,T]} \norm{(I-P_N)A^{-\frac{\chi}{2}}S(t)A^{\frac{\chi}{2}}x} \le \norm{(I-P_N)A^{-\frac{\chi}{2}}} \sup_{t \in [0,T]} \norm{S(t)A^{\frac{\chi}{2}}x} \lesssim \lambda_{N+1}^{-\frac{\chi}{2}}. 
	\end{equation*}
	
	For the second term, we use Proposition~\ref{app:prop:cont-error}, \eqref{eq:schatten-bound-1}, Assumption~\ref{ass:Q}, \eqref{eq:semigroup-bound} and~\eqref{eq:projection-error} to see that, for arbitrary $\epsilon \in (0,\beta-r)$, 
	\begin{align*}
		\mathrm{II}^2 = \norm[L^p(\Omega,\cC_T)]{X^0-X_N^0}^2 &\lesssim \int_{0}^{T} t^{-\epsilon} \norm[\cL_2(Q^{\frac{1}{2}}(H),H)]{(I-P_N)S(t)}^2 \dd t \\
		&\quad+ \sup_{t \in [0,T]} \norm[\cL(\dot{H}^r,H)]{(I-P_N)S(t)}^2 \\
		&\lesssim \lambda_{N+1}^{-r} \int^T_0 t^{\beta-r-\epsilon-1} \dd t + \lambda_{N+1}^{-r} \lesssim \lambda_{N+1}^{-r}.
	\end{align*}
	
	It similarly follows by Lemma~\ref{lem:abstract-error-lem-1} and~\eqref{eq:projection-error} that $
	\mathrm{V} \lesssim \lambda_{N+1}^{-{\min(\rho-\alpha,r)}/{2}}$ for arbitrary $r < \beta$.
	
	Writing $\cE_N := \cK ((Q(T)+\tilde Q)A^\eta)^{-\frac{1}{2}} (X^0(T) + Z) - \cK_N (Q_N(T)+\tilde Q_N)^{-\frac{1}{2}} P_N(X^0(T) + Z)$, we combine Chebyshev's inequality with the bound on $\mathrm{V}$ in a Borel--Cantelli argument. We obtain for all $\epsilon > 0$ and $p > 1$ the existence of a constant $C>0$ such that for all $N \in \N$
	\begin{equation*}
		\IP(\norm[\cC_T]{\cE_N} \ge \lambda_{N+1}^{(-\min(\rho-\alpha,r)+\epsilon)/2}) \le \lambda_{N+1}^{(\min(\rho-\alpha,r)-\epsilon)p/2} \norm[L^p(\Omega,\cC_T)]{\cE_N}^p  \le C \lambda_{N+1}^{-\epsilon p}.
	\end{equation*}
	Since~$I_{\dot{H}^{\zeta} \hookrightarrow H} \in \cL_2(\dot{H}^{\zeta},H)$, $
	\sum_{j = 1}^\infty \lambda_j^{-\zeta} < \infty$.
	Hence, if we choose $ p \ge \zeta/\epsilon $, the corresponding series in the Borel--Cantelli lemma is convergent. Then, $\IP$-a.s., there is some $M \in \N$ such that $
	\norm[\cC_T]{\cE_N}< \lambda_{N+1}^{(-\min(\rho-\alpha,r)+\epsilon)/2}$
	for all $N \ge M$. Let now $\cM_{N,X^0(T)+Z} \subset \dot{H}^\eta$, with $\mu_T(\cM_{N,X^0(T)+Z}) = 1$, be the Borel subspace on which $P_N$ is well-defined (see Remark~\ref{rem:proj-cyl-obs-well-def}). Then, by the same arguments, there is a Borel subset (not necessarily a subspace) $\cM_1 \subset \dot{H}^{\eta}$ with $\mu_T(\cM_1) = 1$ such that for each $y \in \cM_1 \cap_{N\in\N} \cM_{N,X^0(T)+Z}$, there is some $M \in \N$ such that
	\begin{equation*}
		\lrnorm[\cC_T]{\cK ((Q(T)+\tilde Q)A^\eta)^{-\frac{1}{2}} y - \cK_N (Q_N(T)+\tilde Q_N)^{-\frac{1}{2}} P_N y} < \lambda_{N+1}^{(-\min(\rho-\alpha,r)+\epsilon)/2}
	\end{equation*}
	for all $N \ge M$. Let $\cM_0$ be the Borel subspace of Proposition~\ref{prop:bridge-existence}. By letting $\cM := \cM_0 \cap \cM_1 \cap_{N\in\N} \cM_{N,X^0(T)+Z}$, we find that $\mathrm{III} < \lambda_{N+1}^{(-\min(\rho-\alpha,r)+\epsilon)/2}$ for all $N \ge M$. 
	
	It remains to treat the fourth term. We first note that by~\eqref{eq:lem:abstract-error-lem-1:inverse-bound-1} and~\eqref{eq:semigroup-bound}, 
	\begin{equation*}
		\norm[\dot{H}^{\eta}]{((Q(T)+\tilde Q)A^\eta)^{-\frac{1}{2}} S(T) x} \lesssim \norm[\dot{H}^\alpha]{S(T)x} \lesssim T^{-\max((\alpha-\chi),0)/2}
	\end{equation*}
	Writing $\cJ^0_{N,H}$ for the operator $\cJ^0_{V_N,H}$ defined in~\eqref{eq:JVH-def}, we therefore obtain 
	\begin{align*}
		\mathrm{IV} &= \norm[\cC_T]{ \cJ_1 \cJ_0 - P_N \cJ_1 \cJ^0_{N,H} ((Q(T)+\tilde Q)A^\eta)^{-\frac{1}{2}} S(T) x } \\
		&\lesssim \norm[\cL(H_T,\cC_T)]{(I-P_N)\cJ_1} \norm[\cL(\dot{H}^{\eta},H_T)]{\cJ_0} + \norm[\cL(H_T,\cC_T)]{\cJ_1} \norm[\cL(\dot{H}^{\eta},H_T)]{\cJ_0-\cJ^0_{N,H}}.
	\end{align*}
	In a similar way to the first term, we obtain from Assumption~\ref{ass:rho} and H\"older's inequality that 
	$\norm[\cL(H_T,\cC_T)]{(I-P_N)\cJ_1} \lesssim \lambda_N^{-\frac{\rho}{2}} $. 
	The term $\norm[\cL(\dot{H}^{\eta},H_T)]{\cJ_0-\cJ^0_{N,H}}$ has already been treated in the proof of Lemma~\ref{lem:abstract-error-lem-1}, and with this, the proof is completed. 
\end{proof}

\begin{remark}
	\label{rem:spectral-unconditioned-rate}
	For the unconditioned case $y = Z = 0$, the proof can be repeated to see that under the same assumptions, the convergence rate becomes $r < \min(\beta,\xi)$. Hence, the convergence rate is the same if $\rho \ge \alpha + \min(\beta,\xi)$ and since $\rho \ge \beta$, this is in particular fulfilled for $\alpha = 0$.
\end{remark}

\begin{remark}
	%\label{rem:lower-bound}
	In the special case that not only $\tilde Q$ but also $Q$ share an eigenbasis with $A$, we can derive a lower bound on the error $\norm[L^p(\Omega,\cC_T)]{X_N^{x,y} - X^{x,y}}$
	of Theorem~\ref{thm:spectral-convergence} for $p \ge 2$. Writing $(\mu_j)^\infty_{j=1}$ for the eigenvalues of $Q$, we also need to assume that there is some $\nu > 1$ and $C > 0$ such that $\inf_{j} \lambda_j^{-1} \mu_j \ge C j^{-\nu}$. For simplicity, we let $\eta = -\zeta$ and set $x = y = 0$.
	
	Note that $Q(t) e_j = \mu_j (1 - e^{- 2 \lambda_j t})/ (2 \lambda_j) e_j$ and, since $e_j \in \dot{H}^\alpha \hookrightarrow \Cov(X^0(T) + Z)^{1/2}(\dot{H}^\eta)$, that $$K(t) e_j = Q(t)S(T-t) A^{\eta} ((Q(T) + \tilde Q) A^{\eta})^{-\frac{1}{2}} e_j =  \frac{\frac{\mu_j}{2 \lambda_j}  (1 - e^{- 2 \lambda_j t}) e^{- \lambda_j (T-t)} \lambda_j^{\eta/2}}{\sqrt{\frac{\mu_j}{2 \lambda_j}  (1 - e^{- 2 \lambda_j T}) + \tilde \mu_j}} e_j$$ for $t \in [0,T], j \in \N$. Since $(\lambda^{-\eta/2} e_j)_{j=1}^\infty$ is an orthonormal basis of $\dot{H}^\eta$, Lemma~\ref{lem:conditional-lemma} yields
	$$
	\hat{X}^0(t) = \sum_{j=1}^\infty \frac{\frac{\mu_j}{2 \lambda_j^{1 + \eta}}  (1 - e^{- 2 \lambda_j t}) e^{- \lambda_j (T-t)}}{\frac{\mu_j}{2 \lambda_j}  (1 - e^{- 2 \lambda_j T}) + \tilde \mu_j} \inpro[\dot{H}^\eta]{X(T) + Z}{e_j} e_j. 
	$$
	Writing $\hat{X}_N^0(t) := Q_{N}(t) S_N(T-t) (Q_{N}(T) + \tilde Q_N  )^{-1}(X_N^0(T) + P_N Z)$ we similarly obtain from Remark~\ref{rem:proj-cyl-obs-well-def},
	$$
	\hat{X}_N^0(t) = \sum_{j=1}^{N+1} \frac{\frac{\mu_j}{2 \lambda_j^{1 + \eta}}  (1 - e^{- 2 \lambda_j t}) e^{- \lambda_j (T-t)}}{\frac{\mu_j}{2 \lambda_j}  (1 - e^{- 2 \lambda_j T}) + \tilde \mu_j} \inpro[\dot{H}^\eta]{X(T) + Z}{e_j} e_j.
	$$
	It follows that $\hat{X}_N^0 = P_N \hat{X}^0$. By the definition of $\Cov(X^0(T) + Z)$ and $\Cov(X^0(t),X^0(T) + Z)$, $t \in [0,T]$, $$\E[\inpro[\dot{H}^\eta]{X^0(T) + Z}{e_j}^2] = \lambda_j^{2 \eta} \inpro{(Q(T) + \tilde Q) e_j}{e_j} = \frac{\mu_j}{2 \lambda_j^{1 - 2 \eta}}  (1 - e^{- 2 \lambda_j T}) + \lambda_j^{2 \eta} \tilde \mu_j $$ and $$\E[\inpro{X^0(t)}{e_j}\inpro[\dot{H}^\eta]{X^0(T) + Z}{e_j}] = \lambda_j^\eta \inpro{Q(t) S(T-t)e_j}{e_j} = \frac{\mu_j}{2 \lambda_j^{1 - \eta}}  (1 - e^{- 2 \lambda_j t})e^{- \lambda_j (T-t)}.$$
	Now, for arbitrary $t \in (0,T]$, $\norm[L^p(\Omega,\cC_T)]{X_N^{x,y} - X^{x,y}}^2 \ge \norm[L^2(\Omega,\cC_T)]{X_N^{x,y} - X^{x,y}}^2 \ge \norm[L^2(\Omega,H)]{X_N^{x,y}(t) - X^{x,y}(t)}^2$. Therefore, we get a lower bound on $\norm[L^p(\Omega,\cC_T)]{X_N^{0,0} - X^{0,0}}^2$ by first noting that by definition of $X_N^{0,0}$ and $X^{0,0}$,
	\begin{equation*}
		\norm[L^2(\Omega,H)]{X_N^{0,0}(t) - X^{0,0}(t)}^2 = \sum^\infty_{j=N+1} \E\left[\left( \inpro{X^0(t)}{e_j} - \inpro{\hat X^0(t)}{e_j} \right)^2\right]
	\end{equation*}
is given by
	\begin{align*}
		&\sum^\infty_{j=N+1} \frac{\mu_j (1 - e^{- 2 \lambda_j t})}{2 \lambda_j} - 2 \frac{\frac{\mu_j}{2 \lambda_j^{1 + \eta}}  (1 - e^{- 2 \lambda_j t}) e^{- \lambda_j (T-t)}}{\frac{\mu_j}{2 \lambda_j}  (1 - e^{- 2 \lambda_j T}) + \tilde \mu_j} \frac{\mu_j (1 - e^{- 2 \lambda_j t})e^{- \lambda_j (T-t)}}{2 \lambda_j^{1 - \eta}} \\
		&\hspace{3em}+ \left(\frac{\frac{\mu_j}{2 \lambda_j^{1 + \eta}}  (1 - e^{- 2 \lambda_j t}) e^{- \lambda_j (T-t)}}{\frac{\mu_j}{2 \lambda_j}  (1 - e^{- 2 \lambda_j T}) + \tilde \mu_j}\right)^2 \left(\frac{\mu_j}{2 \lambda_j^{1 - 2 \eta}}  (1 - e^{- 2 \lambda_j T}) + \lambda_j^{2 \eta} \tilde \mu_j \right) \\
		&\quad= \sum^\infty_{j=N+1} \frac{\mu_j (1 - e^{- 2 \lambda_j t})}{2 \lambda_j} \left(1 - \frac{e^{-2 \lambda_j (T-t)} - e^{-2 \lambda_j T}}{1 - e^{- 2 \lambda_j T} + \frac{2 \tilde \mu_j \lambda_j}{\mu_j}}\right).
	\end{align*}
	Writing $C_1$ for the constant of Remark~\ref{rem:tilde-mu-spectral}, we have $ \tilde \mu_j \lambda_j\ge C_1 \lambda_j^{1-\alpha}$ for all $j \in \N$. Moreover, as a consequence of Assumption~\ref{ass:rho}, the fact that $\alpha \le \rho$ and our assumption that $\mu_j > 0$ for all $j \in \N$, there is a constant $C_2 > 0$ such that $\lambda_j^{1 - \alpha}/\mu_j \ge (1 - e^{-2\lambda_j T})/C_2$ for all $j \in \N$. Putting these facts together, it follows that for a constant $C > 0$, $\frac{2 \tilde \mu_j \lambda_j}{\mu_j} \ge C (1 - e^{-2 \lambda_j T})$ for all $j \in \N$.
	Since the function $x \mapsto x/(1+x)$ is decreasing,
	$$
	1 - \frac{e^{-2 \lambda_j (T-t)} - e^{-2 \lambda_j T}}{1 - e^{- 2 \lambda_j T} + \frac{2 \tilde \mu_j \lambda_j}{\mu_j}} \ge 1 - \frac{1 - e^{-2 \lambda_j T}}{1 - e^{- 2 \lambda_j T} + \frac{2 \tilde \mu_j \lambda_j}{\mu_j}} = \frac{1}{\frac{1 - e^{- 2 \lambda_j T}}{\frac{2 \tilde \mu_j \lambda_j}{\mu_j}} + 1} \ge \frac{C}{C + 1}. 
	$$
	By our assumption on $\inf_{j} \lambda_j^{-1} \mu_j$, it follows that for some constant $C > 0$, 
	\begin{align*}
		\norm[L^p(\Omega,\cC_T)]{X_N^{x,y} - X^{x,y}}^2 \ge C \sum^\infty_{j=N+1} j^{-\nu} \gtrsim N^{1-\nu}.
	\end{align*}
	where we used the fact that $\lambda_j \to \infty$ as $j \to \infty$ in the first inequality and an integral inequality in the second, cf.\ the proof of \cite[Lemma~2.5]{BGJK20}. This implies that the bound on $\norm[L^p(\Omega,\cC_T)]{X_N^{x,y} - X^{x,y}}^2$ is essentially sharp under appropriate conditions on $\lambda_j$, $j \in \N$. For example, we may consider the setting of space-time white noise (see Example~\ref{ex:Q-reg}) with $\cD$ being an interval in $\R$. Then $\mu_j = 1$ and $\lambda_j$ is proportional to $j^2$ for all $j$, see, e.g., \cite[Chapter~6]{K14}. This means we may, for small $\epsilon > 0$, take $\nu = 2$ and $\beta = 1/2 - \epsilon$, so that $\norm[L^p(\Omega,\cC_T)]{X_N^{x,y} - X^{x,y}}$ is bounded from below by a constant times $\lambda_{N+1}^{-(\beta + \epsilon)/2}$, and from above (see Theorem~\ref{thm:spectral-convergence}) by a constant times $\lambda_{N+1}^{-(\beta-\epsilon)/2}$.
\end{remark}

\section{Finite element approximation under white observation noise}
\label{sec:approximation-fem}

In this section, we apply the results of Section~\ref{sec:approximation} to a finite element approximation of the SPDE bridge $X^{x,y}$, formulated in an abstract way. For a discretization parameter $h \in (0,1]$, let $(V_h)_{h \in (0,1]}$ be a net of finite-dimensional subspaces of $\dot{H}^1$ equipped with the inner product of $H$ and let $P_h \colon \dot{H}^{-1} \to V_h$ denote the associated generalized orthogonal projection. We write, for a cylindrical random variable $Z$ in $H$, $P_h Z$ for the $V_h$-valued random variable defined by Lemma~\ref{lem:fin-dim-proj-well-def}. Let us also write $S_h := S_V$ for a $V_h$-valued approximation of the semigroup $S$, $Q_h := Q_{V_h}$ and $\tilde Q_h := \tilde Q_{V_h}$. We write $X^{x}_h := X^{x}_{V_h}$ and $X^{x,y}_h := X^{x,y}_{V_h}$ for the discrete processes given by~\eqref{eq:abstract-discrete-spde-def} and \eqref{eq:abstract-discrete-bridge-def}. We make the following assumption on $P_h$ and $S_h$. 
\begin{assumption}
	\label{ass:fem}
	The following two statements hold true. 
	\begin{enumerate}[label=(\roman*)]
		\item \label{ass:fem:sg} For all $r \in [0,2]$ and $s \in [-r,\min(1,2-r)]$, there is a constant $C < \infty$ such that for all $h \in (0,1]$ and $t \in [0,T]$,
		\begin{equation*}
			\norm[\cL(\dot{H}^{-s},H)]{S(t)-S_h(t)P_h} = \norm[\cL(H)]{(S(t)-S_h(t)P_h)A^{\frac{s}{2}}} \le C h^r t^{-\frac{r+s}{2}}.
		\end{equation*}
		\item \label{ass:fem:proj} For all $s \in [0,2]$, there is a constant $C < \infty$ such that for all $h \in (0,1]$,
		\begin{equation*}
			%\label{eq:fem-projection}
			\norm[\cL(\dot{H}^s,H)]{I-P_h} = \norm[\cL(H)]{(I-P_h)A^{-\frac{s}{2}}} = \norm[\cL(H)]{A^{-\frac{s}{2}}(I-P_h)} \le C h^{s}.
		\end{equation*}		
	\end{enumerate}
\end{assumption}

\begin{example}
	%\label{ex:fem-example}
	In the setting of Example~\ref{ex:elliptic-operator}, with the additional assumption that $\cD$ is a convex polygon, we consider the same approximation $S_h$ of $S$ as in~\cite{FS91}. We let $(V_h)_{h \in (0,1]} \subset {H}^1$ be a standard family of finite element spaces consisting of piecewise linear polynomials with respect to a regular family of triangulations of $\cD$ with maximal mesh size $h$, vanishing on $\partial \cD$ in the Dirichlet case. We assume elliptic regularity, i.e., that $A^{-1} \in \cL(H,H^2)$. This is true when, e.g., the functions $a_{i,j}$, $i,j = 1, \ldots, d$ are of class $\cC^1(\bar{\cD})$ \cite[Theorems~2.6-2.7]{Y10}. Then, the so called Ritz projector $R_h \colon \dot{H}^1 \to V_h$, defined as the orthogonal projection of $\dot{H}^1$ onto $V_h$ with respect to the inner product of $\dot{H}^1$, satisfies the inequality $\norm[\cL(\dot{H}^r,H)]{I-R_h} \le C h^r$ for $r \in \{1,2\}$, where the constant $C < \infty $ does not depend on $h$ \cite[p. 799]{FS91}. This suffices for \cite[Lemma~3.8]{K14} to be satisfied. Then, an interpolation argument as in \cite[Lemma~5.1]{AKL16} (see, e.g., \cite[Theorem~A.4]{BZ00} for a justification) yields the expression in Assumption~\ref{ass:fem}\ref{ass:fem:sg}. Assumption~\ref{ass:fem}\ref{ass:fem:proj} follows, in this setting, from an estimate on an interpolant operator (see, e.g., \cite[p. 799]{FS91}) combined with an interpolation argument. 
\end{example}

We move on to derive a convergence result for the SPDE $X^{x,y}$. In this setting, we can not hope to find an operator $\tilde Q$ other than $\tilde Q = I$ such that Assumption~\ref{ass:abstract-approx}\ref{ass:abstract-approx:commutativity} is satisfied. We explicitly make this choice in the theorem below. 

\begin{theorem}
	%\label{thm:fem-convergence}
	Let Assumption~\ref{ass:Q}, Assumption~\ref{ass:Z} with $\tilde Q = I$ and Assumption~\ref{ass:fem} be satisfied. Let $\chi > 0$ and let $(h_n)_{n = 1}^\infty \subset (0,1]$ be a sequence fulfilling $h_n = \Op(n^{-\omega})$ for some $\omega > 0$. Then, for a Borel subset $\cM \subset \dot{H}^{-\zeta}$ with $\mu_T(\cM) = 1$ and all $x \in \dot{H}^\chi$, $y \in \cM$, $p \ge 1$  and $r < \min(\beta,\chi)$, there is an $M \in \N$ and a constant $C < \infty$ such that for all $N \ge M$,
	\begin{equation*}
		\norm[L^p(\Omega,\cC_T)]{X_{h_N}^{x,y}-X^{x,y}} \le C h_N^{-\min(r,2)}.
	\end{equation*}  
\end{theorem}
\begin{proof}
	We first note that since $\tilde Q = I$, we have $\eta = -\zeta$ and $\alpha = 0$ in Assumption~\ref{ass:Z}. Then Assumption~\ref{ass:abstract-approx-2}\ref{ass:abstract-approx-2:projection} is trivially fulfilled. Moreover, Assumption~\ref{ass:rho} is satisfied with $\rho < \beta$. Using Assumption~\ref{ass:fem}\ref{ass:fem:sg} with $r=0$ along with~\eqref{eq:semigroup-bound}, we see that for all $s \in [0,1]$, there is a constant $C < \infty$ such that for all $t > 0$ and $h \in (0,1]$,
	\begin{equation*}
		%\label{eq:fem-semigroup-bound}
		\norm[\cL(\dot{H}^{-s},H)]{S_h(t)P_h} = \norm[\cL(H)]{S_h(t)P_h A^{\frac{s}{2}} } \le C t^{-s/2}.
	\end{equation*}
	This shows that Assumption~\ref{ass:abstract-approx}\ref{ass:abstract-approx:findim-bound-1} is satisfied.
	Using this result, it follows that for any $\epsilon \in [0,\min(\beta,1))$, 
	\begin{equation*}
		\sup_{h \in (0,1]} \int^T_0 t^{-\epsilon} \norm[\cL_2(Q^{1/2}(H),H)]{S_h(t)P_h}^2 \dd t < \infty.
	\end{equation*}
	From this, we obtain that Assumptions~\ref{ass:abstract-approx}\ref{ass:abstract-approx:findim-bound-2} and~\ref{ass:abstract-approx-2}\ref{ass:abstract-approx-2:Q-alpha} are satisfied, too. 
	
	Using the results above, we apply Lemmas~\ref{lem:abstract-error-lem-1} and~\ref{lem:abstract-error-lem-2} to see that, for $r < \beta$ and $p \ge 1$, $\norm[L^p(\Omega,\cC_T)]{\E[X^{0} | X^0(T) + Z] - \E[X_h^{0} | X_h^0(T) + P_h Z]}^2$ is bounded by a constant times
	\begin{equation}
		\label{eq:thm:fem-convergence:pf-1}
		 \int_{0}^{T} t^{-\epsilon} \norm[\cL_2(Q^{\frac{1}{2}}(H),H)]{S(t)-S_h(t)P_h}^2 \dd t + \sup_{t \in [0,T]} \norm[\cL(\dot{H}^r,H)]{S(t)-S_h(t)P_h}^2+ \norm[\cL(\dot{H}^{r},H)]{I-P_h}^2,
	\end{equation}
	where we have also made use of the fact that the $\cL(Q^{1/2}(H),H)$-norm is bounded by the $\cL_2(Q^{1/2}(H),H)$-norm for the error term of Lemma~\ref{lem:abstract-error-lem-2}. Using Assumptions~\ref{ass:Q} and~\ref{ass:fem}\ref{ass:fem:sg}, it now follows that, for $r < \beta$ and $\epsilon \in (0,\min(\beta-r,1))$,
	\begin{align*}
		\int_{0}^{T} t^{-\epsilon} \norm[\cL_2(Q^{\frac{1}{2}}(H),H)]{S(t)-S_h(t)P_h}^2 \dd t &\le \hspace{-3pt} \int_{0}^{T} t^{-\epsilon} \norm[\cL(\dot{H}^{\beta-1},H)]{S(t)-S_h(t)P_h}^2 \norm[\cL_2(H)]{A^{\frac{\beta-1}{2}} Q^{\frac{1}{2}}}^2 \dd t \\
		&\lesssim h^{2\min(r,2)} \int_{0}^{T} t^{(\beta - r) - 1 - \epsilon} \dd t.
	\end{align*}
	The two other terms can similarly, using Assumption~\ref{ass:fem}, be bounded by a constant times $h^{2\min(r,2)}$.
	
	The proof of the claim is now similar to the proof of Theorem~\ref{thm:spectral-convergence}, and we make the same split of the error. The treatment of the terms $\mathrm{I}$ and $\mathrm{II}$ are entirely analogous, while~\eqref{eq:thm:fem-convergence:pf-1} is used for the term $\mathrm{V}$. In order to treat the term
	\begin{equation*}
		\mathrm{III} = \norm[\cC_T]{\cK (Q(T)+ \tilde Q)^{-\frac{1}{2}} y - \cK_h (Q_{h}(T)+ {\tilde Q}_{h})^{-\frac{1}{2}} P_{h} y }, 
	\end{equation*}
	where $\cK_h \colon V_h \to C([0,T],V_h)$ is given by
	\begin{equation*}
		(\cK_h v)(t) := Q_{h}(t) S_h(T-t) (Q_{h}(T) + {\tilde Q}_{h})^{-\frac{1}{2}} v,
	\end{equation*}
	for $v \in V_h, t \in [0,T]$,
	we can, as before, use a Borel--Cantelli argument applied to 
	\begin{equation*}
		\norm[\cC_T]{\cK (Q(T)+ \tilde Q)^{-\frac{1}{2}} (X^0(T)+Z) - \cK_h (Q_{h}(T)+ {\tilde Q}_{h})^{-\frac{1}{2}} P_{h} (X^0(T)+Z) }.
	\end{equation*}
	This corresponding $L^p(\Omega,\cC_T)$-bound is split into $\mathrm{III}_a + \mathrm{III}_b$, where
	\begin{equation*}
		\mathrm{III}_a = \norm[L^p(\Omega,\cC_T)]{\E[X^{0} | X^0(T) + Z] - \E[X_h^{0} | P_h (X^0(T) + Z)]}
	\end{equation*} 
	can be treated by Lemma~\ref{lem:abstract-error-lem-1}. The second term is given by 
	\begin{align*}
		\mathrm{III}_b = \|&Q_{V_h,H}(\cdot) S(T-\cdot) (P_{h}Q(T)P_{h}+ {\tilde Q}_{h})^{-1} P_{h} (X(T)+Z) \\
		&\quad- Q_{h}(\cdot) S(T-\cdot) (Q_{h}(T)+ {\tilde Q}_{h})^{-1} P_{h} (X(T)+Z)\|_{L^p(\Omega,\cC_T)}.
	\end{align*} 
	In the notation of the proof of Lemma~\ref{lem:abstract-error-lem-2}, this can be written as 
	\begin{align*}
		\mathrm{III}_b = \norm[L^p(\Omega,\cC_T)]{\cJ^1_{V_h} (\cJ^0_{V_h} - \tilde \cJ^0_{V_h})  P_{h} (X(T)+Z)}, 
	\end{align*} 
	and a bound on this can be found similarly to the proof of this lemma. In the same way as in the proof of Theorem~\ref{thm:spectral-convergence}, we then, for each $\epsilon >0$, find a Borel subset $\cM$ with $\mu_T(\cM) = 1$, such that for each $y \in \cM$, there is some $M \in \N$ such that 
	\begin{equation*}
		\norm[\cC_T]{\cK (Q(T)+ \tilde Q)^{-\frac{1}{2}} y - \cK_{h_N} (Q_{h_N}(T)+ {\tilde Q}_{h_N})^{-\frac{1}{2}} P_{h_N} y } < h_N^{-\min(r,2)+\epsilon}
	\end{equation*}
	for all $N \ge M$. The argument for the final term $\mathrm{IV}$ is similar to that of Theorem~\ref{thm:spectral-convergence}, using some of the estimates in the proof of Lemma~\ref{lem:abstract-error-lem-2}. We omit the details.
\end{proof}

\begin{remark}
	\label{rem:fem-unconditioned-rate}
	As for the spectral approximation, the rate in the unconditioned case $y = Z = 0$ can again be seen to be $r < \min(\beta,\xi)$, which coincides with the white observation noise case. 
\end{remark}

\begin{remark}
	Clearly, the same conclusion holds with $\tilde Q = \epsilon I$ for arbitrary $\epsilon > 0$. However, then the constant appearing in the error bound will depend on $\epsilon$.
\end{remark}

	\appendix
	\section{}
\label{app:sec:regularity}

In this appendix, we derive some technical results not directly related to the approximation of SPDE bridges. The first is a minor extension of \cite[Proposition~B.1]{DPZ14}. This is used throughout the main text to obtain bounds on operator inverses in certain negative norms. We recall that the psudoinverse $A^{-1}$ of an operator $A \in \cL(H,U)$ between two Hilbert spaces $H$ and $U$ is the linear (in general unbounded but closed) operator $A^{-1} := (A|_{\ker(A)^\perp})^{-1} \colon A(H) \to \ker(A)^\perp \subset H$. As a straightforward consequence of the definition, $A^{-1} A = P_{\ker(A)^\perp}$,
the orthogonal projection onto the orthogonal complement of $\ker(A)$.
In the case that $A$ is invertible, its pseudoinverse coincides with the usual inverse operator.

\begin{proposition}
	\label{app:inverse-bound}
	Let $H_1, H_2$ and $U$ be real Hilbert spaces and let $A_1 \in \cL(H_1,U)$, $A_2 \in \cL(H_2,U)$. Then $A_1(H_1) \subset A_2(H_2)$ if and only if there is a constant $C < \infty$ such that $\norm[H_1]{A_1^* u} \le C \norm[H_2]{A_2^* u}$ for all $u \in U$. If this holds, then $\norm[H_2]{A_2^{-1} u} \le C \norm[H_1]{A_1^{-1} u}$ for all $u \in A_1(H_1)$.
\end{proposition}
\begin{proof}
	The first statement is \cite[Proposition~B.1 (i)]{DPZ14}. We assume that $A_2$ is injective as otherwise we can consider its restriction to the orthogonal complement of its kernel. For the second claim, let $A_1(H_1) \subset A_2(H_2)$ and pick $u \in A_1(H_1)$. Then there exist $v \in H_1$, $v' \in H_2$ such that $u = A_1 v = A_2 v'$. We must show that $\norm[H_2]{v'} \le C \norm[H_1]{v}$. Suppose by contradiction that $\norm[H_2]{v'} > C \norm[H_1]{v}$. We have
	\begin{equation*}
		\frac{u}{\norm[H_1]{v}} = \frac{A_1 v}{\norm[H_1]{v}} \in \left\{ A_1 w : \norm[H_1]{w} \le 1 \right\}.
	\end{equation*}
	From the proof of \cite[Proposition~B.1 (i)]{DPZ14}, we see that the set on the right hand side must be contained in 
	\begin{equation*}
		\left\{ A_2 w : \norm[H_2]{w} \le C \right\}.
	\end{equation*}
	But $\frac{u}{\norm[H_1]{v}} = A_2 \frac{v'}{\norm[H_1]{v}}$ with $\norm[H_2]{v'} > C \norm[H_1]{v}$ and $A_2$ is injective, so $\norm[H_2]{v'} \le C \norm[H_1]{v}$. 
\end{proof}

Next, we consider the same setting and notation as in Section~\ref{sec:spde-bridges}. A Burkholder--Davis--Gundy type inequality~\cite[Theorem~4.36]{DPZ14} provides a simple error bound for $X^0(t)-X_V^0(t)$. Namely, for all $p \ge 1$, $t \in [0,T]$, there is a constant $C< \infty$ such that
\begin{equation*}
	%\label{eq:app:bdg-error}
	\E\left[\norm{X^0(t)-X_V^0(t)}^p\right] \le C \left(\int^t_0 \norm[\cL_2(Q^{\frac{1}{2}}(H),H)]{S(s)-S_V(s)P_V}^2 \dd s\right)^{\frac{p}{2}}.
\end{equation*}
However, we need the following stronger result, which is obtained from a factorization argument similar to the proof of~\cite[Theorem~5.12]{DPZ14}. 

\begin{proposition}
	\label{app:prop:cont-error}
	Suppose that Assumptions~\ref{ass:Q} and \ref{ass:abstract-approx}\ref{ass:abstract-approx:findim-bound-1} are satisfied. Then, for all $r < \beta$, $\epsilon \in (0,\min(\beta,1))$ and $p \ge 1$, there is a constant $C< \infty$ such that for all $V \in (V_i)_{i \in \cI}$,
	\begin{align*}
		&\norm[L^p(\Omega,\cC_T)]{X^0-X_V^0} \\
		&\qquad\le C\Bigg( \int_{0}^{T} t^{-\epsilon} \norm[\cL_2(Q^{\frac{1}{2}}(H),H)]{S(t)-S_V(t)P_V}^2 \dd t + \sup_{t \in [0,T]} \norm[\cL(\dot{H}^r,H)]{S(t)-S_V(t)P_V}^2 \Bigg)^{\frac{1}{2}}.
	\end{align*}
\end{proposition}
\begin{proof}
	By~\eqref{eq:Q:cont} with $r=0$, we may use~\cite[Theorem~5.10]{DPZ14} to apply a factorization formula and obtain 
	\begin{equation*}
		X^0(t) = \frac{\sin(\epsilon \pi)}{\pi} \int_{0}^{t} (t-s)^{\epsilon/2-1} S(t-s) Y(s) \dd s
	\end{equation*}
	for $t \in [0,T]$, where 
	\begin{equation*}
		Y(t) := \int^t_0 (t-s)^{-\epsilon/2} S(t-s) \dd W(s)
	\end{equation*}
	for $t \in [0,T]$. An analogous expression, with $S$ replaced by $S_V$, is obtained for $X^0_V$ using the fact that $V$ is finite dimensional, with 
	\begin{equation*}
		Y_V(t) := \int^t_0 (t-s)^{-\epsilon/2} S_V(t-s) P_V \dd W(s)
	\end{equation*}
	for $t \in [0,T]$. Using these expressions, we make the split
	\begin{align*}
		X^0(t)-X_V^0(t)  &= \frac{\sin(\epsilon \pi)}{\pi} \int_{0}^{t} (t-s)^{\epsilon/2-1} S_V(t-s)P_V (Y(s)-Y_V(s)) \dd s \\
		&\quad+ \frac{\sin(\epsilon \pi)}{\pi} \int_{0}^{t} (t-s)^{\epsilon/2-1} (S(t-s) - S_V(t-s)P_V) Y(s) \dd s =: \mathrm{I}_t + \mathrm{II}_t.
	\end{align*}
	We start with the second term, and let $p > 2/\epsilon$ so that its H\"older conjugate $q$ fulfills $q(\epsilon/2-1) > -1$. Thus, by H\"older's inequality, 
	\begin{equation*}
		\norm{\mathrm{II}_t}^p \le \left( \int^t_0 \norm[\cL(\dot{H}^r,H)]{S(t-s)-S_V(t-s)P_V}^q (t-s)^{q(\epsilon/2-1)} \dd s \right)^{\frac{p}{q}} \int^t_0 \norm[\dot{H}^r]{Y(s)}^p \dd s
	\end{equation*}
	so that 
	\begin{equation*}
		\norm[L^p(\Omega,\cC_T)]{\mathrm{II}}^p \lesssim \sup_{t \in [0,T]} \norm[\cL(\dot{H}^r,H)]{S(t)-S_V(t)P_V}^p \int^T_0 \E[\norm[\dot{H}^r]{Y(t)}^p] \dd t.
	\end{equation*}
	By~\cite[Theorem~4.36]{DPZ14}, 
	\begin{equation*}
		\E[\norm[\dot{H}^r]{Y(t)}^p] \lesssim \int^t_0 (t-s)^{-\epsilon} \norm[\cL_2(Q^{\frac{1}{2}}(H),H)]{S(t-s)}^2 \dd s \le \int^T_0 s^{-\epsilon} \norm[\cL_2(Q^{\frac{1}{2}}(H),H)]{S(s)}^2 \dd s
	\end{equation*}
	for $t \in [0,T]$, which, in light of~\eqref{eq:Q:cont}, completes the bound on $\mathrm{II}$. Next, by a similar H\"older argument,
	\begin{equation*}
		\norm[L^p(\Omega,\cC_T)]{\mathrm{I}}^p \lesssim \sup_{\substack{V \in (V_i)_{i \in \cI} \\ t \in [0,T]}} \norm[\cL(H)]{S_V(t)P_V}^p \int^T_0 \E[\norm{Y(t)-Y_V(t)}^p] \dd t.
	\end{equation*}
	Another application of~\cite[Theorem~4.36]{DPZ14} yields that
	\begin{equation*}
		\E[\norm{Y(t)-Y_V(t)}^p] \lesssim \left(\int_{0}^{T} s^{-\epsilon} \norm[\cL_2(Q^{\frac{1}{2}}(H),H)]{S(s)-S_V(s)P_V}^2 \dd s \right)^{\frac{p}{2}},
	\end{equation*} 
	which, along with Assumption~\ref{ass:abstract-approx}\ref{ass:abstract-approx:findim-bound-1}, completes the proof. 
\end{proof}
	
	\bibliographystyle{hplain}
	\bibliography{bibliography}
	
\end{document}